\documentclass[10pt]{amsart}
\usepackage{geometry}
\usepackage{amsfonts,amssymb,verbatim,amsmath,amsthm,latexsym,textcomp,amscd}
\usepackage{latexsym,amsfonts,amssymb,epsfig,verbatim}
\usepackage[mathscr]{euscript}
\usepackage{amsmath,amsthm,amssymb,latexsym,graphics,textcomp}
\usepackage{float} 
\usepackage{graphicx}
\usepackage{color}
\usepackage{url}
\usepackage{enumerate}
\usepackage{faktor}
\usepackage{xcolor}
\RequirePackage[bookmarks, bookmarksopen=true, plainpages=false, pdfpagelabels, pdfpagelayout=SinglePage, breaklinks = true]{hyperref}
\usepackage[noabbrev,capitalize,nameinlink]{cleveref}
\hypersetup{
	colorlinks,
	linkcolor={red!50!black},
	citecolor={green!50!black},
	urlcolor={red!50!yellow}
}
\usepackage{centernot}
\usepackage{mathtools}
\usepackage{microtype}
\usepackage{tikz-cd}
\usepackage{quiver}
\usepackage{pgfplots}
\pgfplotsset{compat=1.15}
\usepackage{mathrsfs}
\usetikzlibrary{arrows}

\newtheorem{theorem}{Theorem}[section]
\newtheorem{prop}[theorem]{Proposition}
\newtheorem{lemma}[theorem]{Lemma}

\newtheorem{definition}[theorem]{Definition}

\newtheorem{rmk}[theorem]{Remark}

\newtheorem{subtheorem}{Theorem}[subsection]
\newtheorem{subprop}[subtheorem]{Proposition}
\newtheorem{sublemma}[subtheorem]{Lemma}

\newtheorem{subdefinition}[subtheorem]{Definition}

\newtheorem{subrmk}[subtheorem]{Remark}

\newcommand{\ANZ}{\mathscr{A}(Z)}
\newcommand\FF{{\mathcal F}}
\newcommand\MM{{\mathcal M}}
\newcommand\OO{{\mathcal O}}
\newcommand\TT{{\mathcal T}}
\newcommand\R{{\mathbb R}}
\newcommand{\diam}{\operatorname{diam}}

\title{Polyhedral structure of maximal Gromov hyperbolic spaces with finite boundary}

\author{Kingshook Biswas}
\address{Stat-Math Unit, Indian Statistical Institute, 203 B. T. Rd., Kolkata 700108, India}
\email{kingshook@isical.ac.in}
\author{Arkajit Pal Choudhury}
\address{Stat-Math Unit, Indian Statistical Institute, 203 B. T. Rd., Kolkata 700108, India}
\email{rkjtpalchoudhury\_r@isical.ac.in}

\subjclass[2020]{Primary 52B99, 51F99; Secondary 51F30}
\keywords{Gromov hyperbolic spaces, Cross-ratio, Quasi-metric antipodal spaces, Injective metric spaces, Polyhedral complex}
\begin{document}
	\begin{abstract} The boundary $\partial X$ of a Gromov hyperbolic space $X$ carries a natural quasi-Moebius structure, induced by the family of visual quasi-metrics on the boundary $\rho_x = e^{-(\cdot|\cdot)_x}, x \in X$, which are all quasi-Moebius equivalent to each other. For boundary continuous Gromov hyperbolic spaces the visual quasi-metrics are Moebius equivalent to each other and one has a finer Moebius structure on the boundary. In this context, a natural problem is to try to reconstruct the space $X$ from its boundary $\partial X$. For a proper, geodesically complete, boundary continuous Gromov hyperbolic space $X$, the boundary $\partial X$ equipped with its cross-ratio is a particular kind of quasi-metric space, called a {\it quasi-metric antipodal space}. Given a quasi-metric antipodal space $Z$, one may consider the family of all hyperbolic fillings of $Z$.  In \cite{biswas2024quasi} it was shown that this family has a unique upper bound $\mathcal{M}(Z)$ (with respect to a natural partial order on hyperbolic fillings of $Z$), which can be described explicitly in terms of the cross-ratio on $Z$. 
		
		\medskip
		
		A proper, geodesically complete, boundary continuous Gromov hyperbolic $X$ is called a maximal Gromov hyperbolic space if it is an upper bound for the family of fillings of its boundary $\partial X$. In \cite{biswas2024quasi}, it was shown that the maximal Gromov hyperbolic spaces are precisely those of the form $\mathcal{M}(Z)$ for some quasi-metric antipodal space $Z$. A natural problem is to describe explicitly the maximal Gromov hyperbolic spaces $X$ whose boundary $\partial X$ is finite. 
		
		\medskip
		
		We show that for a maximal Gromov hyperbolic space $X$ with boundary $\partial X$ of cardinality $n$, the space $X$ is isometric to a finite polyhedral complex embedded in $(\mathbb{R}^n, ||\cdot||_{\infty})$ with cells of dimension at most $n/2$, given by attaching $n$ half-lines to vertices of a compact polyhedral complex. In particular the geometry at infinity of $X$ is trivial. 
		
		\medskip
		
		The combinatorics of the polyhedral complex is determined by certain relations $R \subset \partial X \times \partial X$ on the boundary $\partial X$, called {\it antipodal relations}. An antipodal relation $R$ on $\partial X$ is a relation given, for some $x \in X$, by the set of pairs $(\xi, \eta) \in \partial X \times \partial X$ which are antipodal with respect to $x$, i.e. there is a bi-infinite geodesic in $X$ joining $\xi, \eta$ and passing through $x$. Each cell of the polyhedral complex is given by a maximal collection of ``equi-antipodal" points in $X$, i.e. a maximal collection of points which have the same antipodal relation. Faces $C'$ of a cell $C$ given by an antipodal relation $R$ correspond to antipodal relations $R'$ containing $R$.   
		
		\medskip
		
		In \cite{biswas2024quasi} it was shown that maximal Gromov hyperbolic spaces are injective metric spaces. We give a shorter, simpler proof of this fact in the case of spaces with finite boundary, using the fact that their geometry at infinity is trivial. We also consider the space of deformations of a maximal Gromov hyperbolic space with finite boundary, and define an associated Teichmuller space and mapping class group.  
	\end{abstract}
	
	\bigskip
	
	\maketitle
	
	\tableofcontents
	
	\section{Introduction.}\label{intro}
	
	The Gromov product on the boundary of a Gromov hyperbolic space $X$ gives rise to a family of {\it visual quasi-metrics} on the boundary $\rho_x(.,.) = e^{-(.|.)_x}, x \in X$. The visual quasi-metrics $\rho_x$ are {\it quasi-Moebius equivalent} to each other, in the sense that the cross-ratios of these quasi-metrics $\rho_x$ as $x$ ranges over $X$ differ by a uniformly bounded multiplicative error, and hence the visual quasi-metrics define a {\it quasi-Moebius} structure on the boundary $\partial X$. It is well-known that this quasi-Moebius structure on the boundary determines the quasi-isometry class of the space $X$. More precisely, Jordi \cite{jordi} proved that a power-quasi-Moebius map between boundaries of visual, roughly geodesic, Gromov hyperbolic spaces extends to a quasi-isometry between the spaces (extending previous results of Bonk-Schramm \cite{bonk-schramm} and Buyalo-Schroeder \cite{buyalo2007elements} on quasi-isometric extension of quasi-symmetric and quasi-Moebius maps between boundaries). The proofs of these results rely on the notion of {\it hyperbolic fillings} (see \cite{buyalo2007elements}): given a compact quasi-metric space $Z$, a hyperbolic filling of $Z$ is a pair $(X, f \colon \partial X \to Z)$ where $X$ is a Gromov hyperbolic space and $f$ is a quasi-Moebius homeomorphism.  
	
	\medskip  
	
	For a Gromov hyperbolic space $ X $ which is \textit{boundary continuous}, i.e. the Gromov product on $ X $ extends continuously to the Gromov boundary $ \partial X $ (cf. \cite[Section 3.4.2, p. 32]{buyalo2007elements}),  the visual quasi-metrics are not just quasi-Moebius equivalent, but are in fact {\it Moebius equivalent}, i.e. they all have the same cross-ratio. In this case one has a finer {\it Moebius structure} on the boundary $\partial X$, and it is a natural question whether this finer Moebius structure is sufficient to recover the space $X$ up to isometry and not just quasi-isometry.  This is an important problem: as explained in \cite{biswas2015moebius}, a positive answer to this question for the class of negatively curved Hadamard manifolds would lead to a solution to the longstanding {\it marked length spectrum rigidity conjecture} of Burns and Katok. In \cite{biswas2015moebius} a partial answer to this question was given, namely it was shown that any Moebius homeomorphism $f \colon \partial X \to \partial Y$ between the boundaries of proper, geodesically complete CAT(-1) spaces $X, Y$ extends to a $(1, \log 2)$-quasi-isometry $F \colon X \to Y$ between the spaces. 
	
	\medskip
	
	In \cite{biswas2024quasi} a natural class of Gromov hyperbolic spaces, the {\it maximal Gromov hyperbolic spaces}, was introduced, for which it was shown that the Moebius structure on the boundary determines the space up to isometry. We first consider the class of Gromov hyperbolic spaces which are proper, geodesically complete and boundary continuous, and call such spaces {\it good Gromov hyperbolic spaces}. As shown in \cite{biswas2024quasi}, the boundary of such a space is a particular type of compact space called a {\it quasi-metric antipodal space} (see \Cref{antipodal definition} for the definition of antipodal spaces). Given a quasi-metric antipodal space $Z$, one considers all hyperbolic fillings $(X, f \colon \partial X \to Z)$ of $Z$, where $X$ is a good Gromov hyperbolic space and $f \colon \partial X \to Z$ is a Moebius homeomorphism.  Note that in this case, even the existence of such a filling is nontrivial, as one requires the homeomorphism $f \colon \partial X \to Z$ to be Moebius and not just quasi-Moebius (the fillings of Bonk-Schramm and Buyalo-Schroeder only give quasi-Moebius homeomorphisms). There is a natural partial order on these fillings: we say $(X, f) \leq (Y, g)$ if the Moebius homeomorphism $h = g^{-1} \circ f \colon \partial X \to \partial Y$ extends to an isometric embedding $H \colon X \to Y$ (it is shown in \cite{biswas2024quasi} that this does indeed define a partial order). A good Gromov hyperbolic space $X$ is said to be a maximal Gromov hyperbolic space if it is an upper bound for the family of hyperbolic fillings of its boundary $\partial X$, i.e. $(Y, f \colon \partial Y \to \partial X) \leq (X, id \colon \partial X \to \partial X)$ for all hyperbolic fillings $(Y, f)$ of $\partial X$. 
	
	\medskip 
	
	In \cite{biswas2024quasi}, it was shown that for any quasi-metric antipodal space $Z$, the family of hyperbolic fillings of $Z$ has a unique upper bound $\mathcal{M}(Z)$ (in particular this family is nonempty) which is a maximal Gromov hyperbolic space. In particular, any good Gromov hyperbolic space admits an isometric embedding into a maximal Gromov hyperbolic space. Moreover, there is an explicit description of the space $\mathcal{M}(Z)$: it is given by the space of Moebius equivalent antipodal functions on $Z$ (we refer to \Cref{antipodal} for the definition of antipodal functions and the metric on $\mathcal{M}(Z)$). It was also shown in \cite{biswas2024quasi} that any maximal Gromov hyperbolic space $X$ is isometric to the space $\mathcal{M}(\partial X)$, and so the class of maximal Gromov hyperbolic spaces coincides with the class of spaces $\mathcal{M}(Z)$ as $Z$ ranges over all quasi-metric antipodal spaces. Moreover, any Moebius homeomorphism $f \colon \partial X \to \partial Y$ between boundaries of maximal Gromov hyperbolic spaces extends to an isometry $F \colon X \to Y$ between the spaces, giving an equivalence of categories between maximal Gromov hyperbolic spaces and quasi-metric antipodal spaces. It is shown in \cite{biswas2024quasi} that proper, geodesically complete trees are maximal Gromov hyperbolic spaces whose boundaries are precisely the ultrametric antipodal spaces, and so the above equivalence of categories extends the equivalence between trees and ultrametric spaces due to Beyrer-Schroeder \cite{beyrer-schroeder}. 
	
	\medskip
	
	The main goal of this article is to obtain an explicit description of all maximal Gromov hyperbolic spaces $X$ whose boundary $\partial X$ is a finite set.  We prove the following:
 
	\medskip
	
	\begin{theorem}\label{1st Thm}
		Let $X$ be a maximal Gromov hyperbolic space such that $\partial X$ is finite of cardinality $n$. Then $X$ is isometric to a polyhedral complex $\mathcal{P} \subset \mathbb{R}^n$ embedded in $(\R^n,\|\cdot\|_\infty)$ with finitely many polyhedral cells each of dimension at most $n/2$. The polyhedral complex $\mathcal{P}$ is given by attaching $n$ half-lines to a compact polyhedral complex.
	\end{theorem} 
	
	\medskip
	
	The combinatorics of the polyhedral complex is determined by certain relations $R \subset \partial X \times \partial X$ on the boundary $\partial X$, called {\it antipodal relations}. An antipodal relation $R$ on $\partial X$ is a relation given, for some $x \in X$, by the set of pairs $(\xi, \eta) \in \partial X \times \partial X$ which are antipodal with respect to $x$, i.e. there is a bi-infinite geodesic in $X$ joining $\xi, \eta$ and passing through $x$. Any $x \in X$ gives rise to a corresponding antipodal relation $R_x$ on $\partial X$. Each cell $C$ of the polyhedral complex corresponds to an antipodal relation $R$, namely $C$ is given by the collection of points $x$ in $X$ whose antipodal relation is $R$, $R_x = R$.  Faces $C'$ of a cell $C$ given by an antipodal relation $R$ correspond to antipodal relations $R'$ containing $R$.  In particular, vertices correspond to maximal antipodal relations. Two cells $C_1, C_2$ corresponding to relations $R_1, R_2$ share a common face $C$ whenever there is an antipodal relation $R$ containing $R_1 \cup R_2$.  
	
	\medskip
	
	The above Theorem has another consequence which is not obvious a priori, namely that the geometry at infinity of any maximal Gromov hyperbolic space with finite boundary is ``trivial", in the sense that outside a compact the space is just a finite union of geodesic rays. This is later used to give a simpler proof that maximal Gromov hyperbolic spaces with finite boundary are injective (see below). 
	
	\medskip
	
	In \cite{biswas2024quasi}, a close connection between maximal Gromov hyperbolic spaces and another well-known class of metric spaces was described, namely the {\it injective metric spaces}. We recall the definition of injective metric spaces 
	(also known as {\it hyperconvex spaces}), first introduced by Aronszajn and Panitchpakdi \cite{aronszajn1956extension}. A metric space $X$ is called injective if for any metric space $A$ and any subspace $B\subset A$ any $1-$Lipschitz map $f \colon B\to X$ has a $1-$Lipschitz extension $F \colon A\to X$ such that $F|_B=f$ (see \cite[Chapter 3]{petrunin2023pure}, \cite[Section 2]{lang2013injective}). Examples of injective spaces are:  the real line, closed intervals in the real line, geodesically complete metric trees and $(\R^n,\|\cdot\|_\infty)$ (more generally $l^\infty(\mathcal K)$ for any arbitrary set $\mathcal K$). There is a well-known characterization of injective metric spaces, which we shall state here without proof.
	
	\medskip
	
	\begin{theorem}[see \cite{aronszajn1956extension}, \cite{lang2013injective}]\label{hyperconvex thm}
		Let $X$ be a metric space. Then the following are equivalent:
		\begin{enumerate}[{\it (1)}]
			\item X is an injective metric space.\medskip
			
			\item If $i \colon X\to Y$ is an isometric embedding into any metric space $Y$ there exists a 1-Lipschitz map $\pi \colon Y\to X$ such that $\pi\circ i= id|_X$, i.e. $X$ is an `absolute 1-Lipschitz retract'.\medskip
			
			\item If $\{B(x_i,r_i)\}_{i\in I}$ is a family of closed balls in $X$ such that 
			\begin{equation}\label{intersection}
				r_i+r_j\ge d(x_i,x_j)
			\end{equation} for any $i,j\in I$ then 
			$\cap_{i\in I}B(x_i,r_i)\neq \emptyset$, i.e. $X$ is `hyperconvex'.\label{hyp}
		\end{enumerate}
	\end{theorem}
	
	\medskip
	
	Isbell \cite{isbell1964six} showed that for any metric space $X$, there is a unique ``smallest" injective space $E(X)$ into which $X$ embeds isometrically (here ``smallest" means that any isometric embedding of $X$ into an injective metric space factors through $E(X)$).  The metric space $E(X)$ is called the injective hull of $X$. This construction was later studied by Dress \cite{dress1984trees} and called it the ``tight-span" of a metric space $X$. An exposition connecting all these articles can be found in \cite{lang2013injective}. For a finite metric space $X$, it is known that the injective hull $E(X)$ is a polyhedral complex of dimension at most $\#X/2$ (for a proof, see for example \cite{lang2013injective}). For further details on the connection between polyhedral structures and injective spaces, we refer to \cite{pavon2016injective}, \cite{pavon2016geometry}, \cite{huber2019polytopal}. Injective hulls or the tight-span of finite metric spaces have appeared in different areas of mathematics such as phylogenetic analysis (see \cite{dress2002explicit}, \cite{dress1996Ttheory}), network flow theory (\cite{dress1998aa}), and tropical geometry (\cite{develin2004}).
	
	\medskip
	
	In \cite{biswas2024quasi}, it was shown that for a good Gromov hyperbolic space $X$, the injective hull $E(X)$ is naturally isometric to the maximal Gromov hyperbolic space $\mathcal{M}(\partial X)$, and the space $X$ itself is injective if and only if it is maximal. In particular, amongst the class of good Gromov hyperbolic spaces, the injective ones, being maximal, are completely determined by the Moebius structures on their boundaries. Furthermore, all maximal Gromov hyperbolic spaces are injective. 
	
	\medskip
	
	In the case of maximal Gromov hyperbolic spaces with finite boundary, we use the fact that the geometry at infinity is ``trivial" in order to give a simpler proof of this last fact in \Cref{Injective metric space and MZ}. In this case, a maximal Gromov hyperbolic space $X$ with finite boundary can be seen as a limit of injective hulls of finite metric spaces: 
	fixing a base-point $x \in X$, for all $r > 0$ large enough, the sphere $S = S(x, r)$ is finite of cardinality $n = \#\partial X$ \footnote{Here `$\#$' denotes cardinality, so $\#\partial X$ denotes the cardinality of the set $\partial X$}, and the closed ball $B(x, r)$ is isometric to the injective hull $E(S)$ of the sphere $S$. The polyhedral structure of $X$ is given by attaching $n$ Euclidean half-lines to the compact polyhedral complex $E(S)$ (as mentioned above, the injective hull of the finite metric space $S$ carries a natural polyhedral structure). Informally speaking, the space $X$ may be thought of as the ``injective hull" of the boundary at infinity $\partial X$. 
	
	\medskip
	
	In \Cref{Teich}, we consider metric deformations of maximal Gromov hyperbolic spaces with finite boundary, and define the Teichmuller space and mapping class group of such a maximal Gromov hyperbolic space. We show that the set of all Moebius equivalence classes of antipodal functions on a finite set $Z$ of cardinality $m$ can be parametrized by an open simplex of dimension $m(m-3)/2$ (see \Cref{parametrization} and \Cref{big-Teich theorem}). Considering the different homeomorphism types of the maximal Gromov hyperbolic spaces $\mathcal{M}(Z, \rho)$ as $\rho$ ranges over all antipodal functions on $Z$ gives a decomposition of this open simplex (the ``big Teichmuller space") into a finite disjoint union of Teichmuller spaces (one for each homeomorphism type, there are finitely many such). In \Cref{4 points} we describe this decomposition and the different homeomorphism types for the case when $Z$ has cardinality equal to $4$. Finally, in \Cref{problems}, we conclude with a brief discussion of some open problems.   
	
	\bigskip
	
	\section{Preliminaries and definitions}
	\noindent In this section, we shall recall some preliminaries and prove some lemmas that we shall use later in the paper.
	\medskip
	
	\subsection{Antipodal spaces and their associated Moebius spaces}\label{antipodal}
	This subsection briefly recalls definitions and results developed in \cite{biswas2024quasi}. We refer the reader to \cite{biswas2024quasi} for a more comprehensive discussion.
	
	\begin{subdefinition}{\bf (Separating function)}
		Let $Z$ be a compact metrizable space with at least four points.
		We call a function $\rho_0\colon Z\times Z\to [0,\infty)$ is separating if:
		\medskip
		
		(1) $\rho_0$ is continuous
		\medskip
		
		(2) $\rho_0$ is symmetric, i.e. $\rho(\xi,\eta)=\rho(\eta,\xi)$ for all $\xi,\eta\in Z$
		\medskip
		
		(3) $\rho_0$ satisfies positivity, i.e. $\rho(\xi,\eta)=0$ if and only if $\xi=\eta$
		
		For our discussion, we shall call the pair $(Z,\rho_0)$ a compact `semi-metric' space, where $\rho_0$ is a separating function defined on a compact metrizable space $Z$.
	\end{subdefinition}
	\medskip
	
	\noindent The `cross-ratio' with respect to a separating function $\rho_0$ for a quadruple of distinct points
	$\xi,\xi',\eta,\eta' \in Z$ is defined to be
	\begin{equation}\label{cross-ratio}
		[\xi,\xi',\eta,\eta']_{\rho_0}\coloneqq
		\frac{\rho_0(\xi,\eta)\rho_0(\xi',\eta')}{\rho_0(\xi,\eta')\rho_0(\xi',\eta)}
	\end{equation}
	
	\noindent Two separating functions $\rho_0,\rho_1$ on $Z$ are said to be {\it Moebius equivalent} if both have the same cross-ratios i.e.
	\begin{equation}\label{same cross-ratio}
		[\xi, \xi', \eta, \eta']_{\rho_0} = [\xi, \xi', \eta, \eta']_{\rho_1}
	\end{equation}
	for all distinct $\xi, \xi', \eta, \eta' \in Z$.
	
	\begin{sublemma}[{\bf G}eometric {\bf M}ean-{\bf V}alue {\bf T}heorem, \cite{biswas2024quasi}]\label{GMVT}
		Two separating functions $\rho_0, \rho_1$ are Moebius equivalent if and only if there exists a positive continuous function $\phi  \colon  Z \to (0, \infty)$ such that,
		$$
		\rho_1(\xi, \eta)^2 = \phi(\xi) \phi(\eta) \rho_0(\xi, \eta)^2
		$$
		for all $\xi, \eta \in Z$.
		Such a function $\phi$ is unique. It is called the derivative of $\rho_1$ with respect to $\rho_0$ and is denoted by $\frac{d\rho_1}{d\rho_0}$. 
	\end{sublemma}
	
	\medskip
	
	\noindent Next, we define the Unrestricted Moebius space as a consequence of GMVT.
	
	\begin{subdefinition}[{\bf Unrestricted Moebius space of a separating function}, \cite{biswas2024quasi}]
		Let $\rho_0$ be a separating function on a compact metrizable space $Z$ with at least four points. The Unrestricted Moebius space $\mathcal{UM}(Z,\rho_0)$ of the separating function $\rho_0$ is the metric space
		\begin{equation*}
			\mathcal{UM}(Z,\rho_0) \coloneqq \{\ \rho \ | \ \rho \ \text{is a separating function Moebius equivalent to} \ \rho_0 \},
		\end{equation*}
		equipped with the metric
		\begin{equation*}
			d_{\mathcal{M}}( \rho_1, \rho_2 )\coloneqq \left|\left| \log \frac{d\rho_2}{d\rho_1} \right|\right|_{\infty},
		\end{equation*} 
		for $\rho_1, \rho_2 \in \mathcal{UM}(Z,\rho_0)$.
		\medskip	
	\end{subdefinition}
	
	We also have the map
	\begin{equation}\label{embedding}
		\begin{split}
			i_{\rho_0} \colon (\mathcal{UM}(Z,\rho_0), d_{\mathcal{M}})& \to ( C(Z),\|\cdot\|_{\infty} ) \\
			\rho & \mapsto \log \frac{d\rho}{d\rho_0}
		\end{split}
	\end{equation}
	which is a  surjective isometry with the inverse given by the map \label{isometry}
	\begin{align*}
		E_{\rho_0} \colon ( C(Z), \|\cdot\|_{\infty} ) & \to  (\mathcal{UM}(Z,\rho_0), d_{\mathcal{M}})\\
		\tau & \mapsto E_{\rho_0}(\tau)
	\end{align*}
	where for $\tau \in C(Z)$, the separating function $\rho = E_{\rho_0}(\tau) \in \mathcal{UM}(Z)$ is defined by
	\begin{equation*}
		\rho(\xi, \eta)^2 \coloneqq e^{\tau(\xi)} e^{\tau(\eta)} \rho_0(\xi, \eta)^2
	\end{equation*}
	for all $\xi, \eta  \in Z$. Moreover, $\log \frac{d\rho}{d\rho_0} = \tau$.
	\medskip
	
	\noindent Now we define :
	
	\begin{subdefinition}{\bf (Antipodal function)}\label{antipodal definition}
		Let $(Z,\rho_0)$ be a compact semi-metric space. The separating function $\rho_0$ is called an antipodal function if it satisfies the following two properties:\medskip
		
		(1) $\rho_0$ has diameter one, i.e. $\sup_{\xi, \eta \in Z} \rho_0(\xi, \eta) = 1$.\medskip
		
		(2) $\rho_0$ is antipodal, i.e. for all $\xi \in Z$ there exists $\eta \in Z$ such that $\rho_0(\xi, \eta) = 1$.\medskip
		
		\noindent We call a compact semi-metric spaces $(Z,\rho_0)$ an antipodal space, if $\rho_0$ is an antipodal function.
		
	\end{subdefinition}
	\medskip
	
	\noindent A {\it quasi-metric antipodal space} is an antipodal space where the antipodal function is a quasi-metric. 
	\medskip
	
	\noindent Here is the definition of the space of our interest:
	
	\begin{subdefinition}[{\bf Moebius space of antipodal function}, \cite{biswas2024quasi}]\label{Moebius space}
		Let $(Z,\rho_0)$ be an antipodal space. The Moebius space $\mathcal{M}(Z,\rho_0)$ is defined to be the metric space
		$$
		\mathcal{M}(Z,\rho_0) \coloneqq \{ \rho \in \mathcal{UM}(Z,\rho_0) \ | \ \rho \ \text{is an antipodal function} \ \}
		$$
		(equipped with the metric $d_{\mathcal{M}}$ as defined above).
	\end{subdefinition}
	
	\noindent We have observed $\mathcal{UM}(Z,\rho_0)$ is isometric to $(C(Z),\|\cdot\|_\infty)$, hence it is proper if and only if $Z$ is finite. However, $\mathcal{M}(Z,\rho_0)$ is always a proper metric space (i.e. closed bounded balls are compact) and a closed subset of $\mathcal{UM}(Z,\rho_0)$ (see  \cite[Lemma 2.7]{biswas2024quasi}). The restriction of the isometry $i_{\rho_0}$ to $\mathcal{M}(Z,\rho_0)$ defines an isometric embedding \begin{equation*}
		i_{\rho_0} \colon \MM(Z,\rho_0)\hookrightarrow(C(Z),\|\cdot\|_\infty)
	\end{equation*} onto a closed subset of $C(Z)$. For $Z$ finite set of cardinality $n$,  $\MM(Z,\rho_0)$ embeds isometrically into $(\R^n,\|\cdot\|_\infty)$.
	\medskip
	
	A \textit{geodesic} in a metric space $(X, d)$ is a path $\gamma \colon I \subseteq \mathbb{R} \to X$ that is an isometric embedding of an interval $I$ into $X$. We say that a metric space $X$ is \textit{geodesic} if any two points $a,b\in X$ can be connected by a geodesic segment, i.e. there exists a geodesic $\gamma\colon [0,d(a,b)] \to X$, with $\gamma(0)=a$ and $\gamma(d(a,b))=b$. Furthermore, a geodesic metric space is called \textit{geodesically complete}, if every geodesic $\gamma \colon I\subset \R\to X$ can be extended to a bi-infinite geodesic $\tilde{\gamma}\colon \R\to X$.
	
	\medskip 
	
	We recall the following properties of $\MM(Z,\rho_0)$ from  \cite{biswas2024quasi} -
	
	\begin{subprop}[Biswas, \cite{biswas2024quasi}]\label{good}
		The space $\mathcal{M}(Z,\rho_0)$ is 
		\medskip
		
		(1) unbounded
		\medskip
		
		(2) contractible 
		\medskip
		
		(3) geodesic
		\medskip
		
		(4) geodesically complete
	\end{subprop}
	
	\medskip
	
	\subsection{Gromov hyperbolicity and the Moebius spaces}\label{Gromov ip}\hfill\\
	Here, we briefly review the essential preliminaries of Gromov hyperbolic spaces. For a more detailed exposition, we refer the reader to \cite[Chapter III]{bridson1999metric} and \cite{buyalo2007elements}. Following this, we will discuss the connection between Gromov hyperbolicity and Moebius spaces.
	
	
	Given a metric space $X$, for any point $x\in X$ we define the Gromov product based at $x$,
	\begin{equation*}
		(y|z)_x\coloneqq \frac{1}{2}(d(y,x)+d(z,x)-d(y,z)).
	\end{equation*}
	A geodesic metric space $X$ is called {\it Gromov hyperbolic} if there exists $\delta \ge 0$ such that every triangle in $X$ is $\delta$-slim (in the terminology of \cite{bridson1999metric}), meaning that each side of the triangle is contained within the $\delta$-neighborhood of the union of the other two sides. Equivalently, a geodesic metric space $X$ is {\it Gromov hyperbolic} if there exists $\delta \ge 0$ such that for any base point $x \in X$ and for all $w, y, z \in X$, the {\it $\delta$-inequality} holds:
	\begin{equation*}
		(y|z)_x \ge \min\{(x|w)_x, (w|y)_x\} - \delta.
	\end{equation*}
	
	Given a proper Gromov hyperbolic space $X$, the Gromov boundary $\partial X$ is defined as equivalence classes of geodesic rays, $\partial X\coloneqq \{\ [\gamma]\ |\ \gamma \colon [0,+\infty)\to X, \hbox{ geodesic ray}\ \}$, where two geodesic rays $\gamma_1$ and $\gamma_2$ are said to be equivalent if and only if the set $\{\ d(\gamma_1(t),\gamma_2(t))\ |\ t\ge 0\}$ is bounded. Consider the disjoint union $\overline{X} \coloneqq X \cup \partial X$, equipped with the \textit{cone topology}\footnote{For a precise definition of cone topology, see \cite[Chapter III.H.3.5]{bridson1999metric}.}, which provides a Hausdorff compactification of $X$ known as the \textit{Gromov compactification}. Given any $x \in X$ and $\xi \in \partial X$, there exists a geodesic ray $\gamma \colon [0, +\infty) \to X$ starting from $x$ and converging to $\xi$, i.e., $\gamma(0) = x$ and $\gamma(t) \to \xi$ in $\overline{X}$ as $t \to +\infty$. Additionally, for distinct points $\xi, \eta \in \partial X$, there exists a bi-infinite geodesic $\gamma \colon \mathbb{R} \to X$ joining $\xi$ to $\eta$, i.e., $\gamma(-t) \to \xi$ and $\gamma(t) \to \eta$ in $\overline{X}$ as $t \to +\infty$. 
	
	Fix any base point $x\in X$, a Gromov hyperbolic space. In general the Gromov product function $(\cdot|\cdot)_x\colon X\times X\to [0,+\infty)$, doesn't extend to a continuous function $(\cdot|\cdot)_x\colon \overline X\times \overline X\to [0,+\infty]$, if it does, we say $X$ is a {\it boundary continuous} (Gromov hyperbolic) space. Proper CAT$(-1)$ spaces are boundary continuous (\cite[Proposition 3.4.2.]{buyalo2007elements}). Biswas showed that proper, geodesically complete CAT$(0)$ Gromov hyperbolic spaces form a large class of examples for boundary continuous Gromov hyperbolic spaces (\cite[Proposition 5.15]{biswas2024quasi}). For Gromov hyperbolic spaces, in general Gromov product is defined for $\xi, \eta \in \partial X$ as
	\begin{equation*}
		(\xi|\eta)_x \coloneqq \sup \limsup_{y_n\to \xi, z_n\to \eta} (y_n|z_n)_x \ \in [0, +\infty]
	\end{equation*}
	supremum taken over all sequences $\{y_n\}$, $\{z_n\}$ in $X$ converging to $\xi$ and $\eta$ respectively in $\overline{X}$ (and $(\xi|\eta)_x=+\infty$ if and only if $\xi = \eta$). 
	
	Now we shall see the following fact about Gromov hyperbolic space which will be useful.
	
	\begin{sublemma}\label{bddlema}
		Let $X$ be a $\delta$-hyperbolic space. Let $\xi,\xi',\eta,\eta'$ be four distinct elements of $\partial X$. Define $B(\xi,\eta), B(\xi',\eta')\subset X$ where 
		\begin{equation*}
			\begin{split}
				&B(\xi,\eta)\coloneqq\{x\in X\ |\ x \text{ is on some bi-infinite geodesic joining }\xi\text{ to }\eta\ \}\\
				&B(\xi',\eta')\coloneqq\{x\in X\ |\ x \text{ is on some bi-infinite geodesic joining }\xi'\text{ to }\eta'\ \}
			\end{split}
		\end{equation*}
		Then $B(\xi,\eta)\cap B(\xi',\eta')$ is bounded.
	\end{sublemma}
	\begin{proof}
		Let $x_0\in B(\xi,\eta)\cap B(\xi',\eta')$. Let $\gamma_1$ be a bi-infinite geodesic joining $\xi$ and $\eta$ which contains $x_0$ and $\gamma_2$ a bi-infinite geodesic joining $\xi'$ and $\eta'$ which contains $x_0$. In a $\delta$-hyperbolic space, any two bi-infinite geodesics joining the same pair of points in the boundary are contained in the $4\delta$-neighborhood of each other.  Thus, if $x\in B(\xi,\eta)\cap B(\xi',\eta')$ then there exist points $a\in \gamma_1$ and $b\in \gamma_2$ such that $|xa|,|xb|<4\delta$. Since $\xi,\xi',\eta,\eta'$ are all distinct the Gromov products $(\xi|\xi')_{x_0},(\xi|\eta')_{x_0},(\xi'|\eta)_{x_0},(\xi'|\eta')_{x_0}$ are all finite. Now, Gromov products are increasing along geodesics, from which it is straightforward to see that
		\begin{align*}
			(a|b)_{x_0}&\le M=\max\{(\xi|\xi')_{x_0},(\xi|\eta')_{x_0},(\xi'|\eta)_{x_0},(\xi'|\eta')_{x_0}\}+2\delta.
		\end{align*}
		This implies,
		\begin{align*}
			d(a,x_0)+d(b,x_0)&\le 2M+d(a,b)\\
			&\le 2M + 8\delta
		\end{align*}
		And hence $|xx_0|\le|xa|+|ax_0|\le 2M+9\delta$, thus $B(\xi,\eta)\cap B(\xi',\eta')$ is bounded.
	\end{proof}
	
	Using the Gromov product on the boundary $\partial X$ with respect to a base point $x\in X$, we define a visual quasi-metric $\rho_x(\cdot, \cdot) \coloneq e^{-(\cdot|\cdot)_x}$ on $\partial X$. In the case of proper, boundary-continuous Gromov hyperbolic spaces $X$, the visual quasi-metrics $\rho_x$ on $\partial X$, for any base point $x \in X$, are Moebius equivalent and define a canonical Moebius structure on $\partial X$ (see \cite[Section 5.3]{buyalo2007elements}). To simplify, we refer to a Gromov hyperbolic space $X$ as ``good'' if it is proper, boundary continuous, geodesic, and geodesically complete.
	For a good Gromov hyperbolic space $X$, the visual quasi-metric spaces $(\partial X,\rho_x)$ are antipodal spaces. Examples of good Gromov hyperbolic spaces include proper, geodesically complete CAT$(-1)$ spaces. Given an isometric embedding $F\colon Y\to X$ between two good Gromov hyperbolic spaces, it extends to a map $\tilde{\Phi}\colon \overline{Y}\to \overline{X}$ between the Gromov compactifications, with $\tilde{\Phi}(\partial Y)\subseteq \partial X$ and the boundary map $$\phi\coloneqq \tilde{\Phi}|_{\partial Y}\colon \partial Y \to \partial X$$ is Moebius map, i.e. $\phi$ preserve cross ratios. If $\Phi$ is an isometry then $\phi$ is a Moebius homeomorphism (cf. \cite[Lemma 5.7]{biswas2024quasi}).

	The following theorem provides a characterization of when Moebius spaces are Gromov hyperbolic.
	\begin{subtheorem}[Biswas, \cite{biswas2024quasi}]
		Let $(Z,\rho_0)$ be an antipodal space. $(Z,\rho_0)$ is a quasi-metric space if and only if $\MM(Z,\rho_0)$ is a Gromov hyperbolic space. 
	\end{subtheorem}
	\noindent If $(Z,\rho_0)$ is an antipodal space with finite elements, then $\rho_0$ is a quasi-metric antipodal space and hence $\MM(Z,\rho_0)$ is a Gromov hyperbolic space (with boundary identified with $Z$). 
	Next, we make the following crucial remark (see \cite[Sections 5 and 6]{biswas2024quasi} for detailed discussion).
	
	\begin{subrmk}\label{gromov ip}
		\noindent For $(Z,\rho_0)$ a quasi-metric antipopdal space, $\MM(Z,\rho_0)$ is a good Gromov hyperbolic space (see \Cref{good}).  The Gromov boundary $\partial \MM(Z,\rho_0)$ is identified with $Z$, and $\MM(Z,\rho_0)$ is  boundary continuous. Furthermore, under this identification for distinct $\xi,\eta\in Z=\partial\MM(Z)$ we have 
		\begin{equation*}
			\exp(-(\xi|\eta)_{\rho_0})=\lim_{\substack{\alpha\to\xi \\ \beta\to \eta}}\ \exp(-(\alpha|\beta)_{\rho_0})\ =\ \rho_0(\xi,\eta),
		\end{equation*}
		where $\alpha,\beta\in \MM(Z,\rho_0)$ (see \cite[Theorem 6.3]{biswas2024quasi}). Therefore, the Gromov boundary $\partial \MM(Z,\rho_0)$ with visual quasi-metric based at $\rho_0$ is identified with $(Z,\rho_0)$. 
	\end{subrmk}
	
	Additionally, the Moebius space $\MM(Z,\rho_0)$ associated to any quasi-metric antipodal space is a {\bf maximal Gromov hyperbolic spaces} in the sense that: if a good Gromov hyperbolic $X$ admits a Moebius homeomorphism $\phi\colon \partial X\to (Z,\rho_0)$ (i.e. $X$ is a hyperbolic filling of $(Z,\rho_0)$), then there exists an isometric embedding $\Phi\colon X\to \MM(Z,\rho_0)$ (see \cite[Theorem 1.2, 1.3]{biswas2024quasi}) such that the boundary map of $\Phi$ agrees with $\phi$. It is the unique maximal among all hyperbolic fillings of $(Z,\rho_0)$ as discussed in \Cref{intro}. Furthermore, any maximal Gromov hyperbolic space $X$ is isometric to the Moebius space $\MM(\partial X,\rho_x)$, where $(\partial X,\rho_x)$ is a visual antipodal space for any $x\in X$.
	
	For a detailed discussion on hyperbolic filling and maximal Gromov hyperbolic spaces, we refer the reader to \cite[Section 6]{biswas2024quasi}. We also encourage exploring the general concepts of Gromov Product spaces and maximal Gromov Product Spaces in \cite[Section 5]{biswas2024quasi}.
	
	\medskip
	
	We now recall the definition of tangent space of $\MM(Z,\rho_0)$.
	\begin{subdefinition}[{\bf Tangent space of $\MM(Z,\rho_0)$}, \cite{biswas2024quasi}]
		Let $\rho\in\MM(Z,\rho_0)$. We say that curve $t\in(-\epsilon,\epsilon)\mapsto\rho(t)\in \MM(Z,\rho_0)$ is admissible at $\rho$ if $\rho(0)=\rho$ and $\frac{d}{dt}_{|_{t=0}}i_{\rho_0}(\rho(t))\in C(Z)$ exits. The tangent space $T_\rho\MM(Z,\rho_0)$ of $\MM(Z,\rho_0)$ at $\rho$ is defined to be a subset of $C(Z)$,
		\begin{equation*}
			T_\rho\MM(Z,\rho_0)\coloneqq\biggl\{\tau=\frac{d}{dt}_{|_{t=0}}i_{\rho_0}(\rho(t))\in C(Z)\ |\ t\mapsto \rho(t)\ \text{is an admissible curve}\ \biggr\}
		\end{equation*}
	\end{subdefinition}
	\noindent From the definition, it is unclear if the tangent space $T_\rho\MM(Z,\rho_0)$ is a vector space. The following theorem makes it clear.
	
	\begin{subtheorem}[Biswas, \cite{biswas2024quasi}]\label{tangent}
		Let $\rho\in \MM(Z,\rho_0)$. A continuous function $\tau\in C(Z)$ is said to be $\rho$-odd if for all $\xi,\eta\in Z$, whenever $\rho(\xi,\eta)=1$ then $\tau(\xi)+\tau(\eta)=0$. $\OO_\rho(Z)\subset C(Z)$ denote the set of $\rho$-odd functions. The tangent space to $\MM(Z,\rho_0)$ at $\rho$ is equal to the space of $\rho$-odd functions, i.e.
		$$T_\rho\MM(Z,\rho_0)=\OO_\rho(Z).$$
		In particular $T_\rho\MM(Z,\rho_0)$ is a closed linear subspace of $C(Z)$.
	\end{subtheorem}
	
	\medskip
	
	\subsection{Polyhedral complex}
	For our purpose, we shall use the following definitions
	\begin{subdefinition}{\bf (Polyhedral cell or convex polyhedron)}
		A non-empty closed convex set $X\subset \R^n$ is called a polyhedral cell (or convex polyhedron) if it is the intersection of a finite number of closed half-spaces.
	\end{subdefinition}
	
	\noindent  If the polyhedron doesn't contain any ray, then it is bounded and called a {\it polytope}. We know that a subset $X\in \R^n$ is a polytope if and only if it is the convex hull of finitely many points in $\R^n$(see \cite[Theorem 1.1]{ziegler2012lectures}). In this article, whenever we say {\it ``polyhedron"}, we shall mean a ``convex polyhedron" (or a ``polyhedral cell"). The dimension of a polyhedral cell is the dimension of the smallest affine subspace that contains it.
	
	\begin{subdefinition}{\bf (Face of a polyhedron)}
		Let $X$ be a polyhedron. Let $H$ be a hyperplane in $\R^n$ such that $X$ is contained in one of the two closed half-spaces bounded by $H$. If $H\cap X\neq \emptyset$ then $H\cap X$ is called a face of $X$.
		
		\noindent Also, if $H\cap X\neq X$ then it is called a `proper face', and if $H\cap X$ is a singleton then it is called a `vertex'. 
	\end{subdefinition} 
	
	\noindent A face of a polyhedron is itself a polyhedral cell. For a polytope, the number of faces is finite, and the intersection of any two faces is also a face (see  \cite[Section 7.34]{bridson1999metric} and  \cite[Chapter 2]{ziegler2012lectures}). Similar arguments would show that the same is true for general polyhedral cells (i.e., they are not necessarily bounded).  Next, we recall the definition of a polyhedral complex (cf.   \cite[Definition 3.13, p. 136]{aleksandrov1998combinatorial}).
	
	\medskip
	
	\begin{subdefinition}{\bf (Polyhedral complex), \cite{ziegler2012lectures}}
		A polyhedral complex $\mathcal C$ is a set of polyhedra in $\R^n$ that satisfies the following:
		\medskip
		
		(1) $\emptyset\in \mathcal C$.
		\medskip
		
		(2) Every face of a polyhedron from $\mathcal C$ is also in $\mathcal C$.
		\medskip
		
		(3) If $P_1,P_2\in \mathcal C$ then $P_1\cap P_2$ is a face of both $P_1$ and $P_2$.
		
	\end{subdefinition}
	\noindent Note $\emptyset$ is a face of every polyhedron, and the intersection of two polyhedra in $\mathcal C$ can be empty. The {\it underlying set} of $\mathcal C$ is the point set $|\mathcal C|\coloneqq \bigcup_{P\in \mathcal C} P$.
	\bigskip
	
	\subsection{Injective metric spaces and injective hulls} Injective metric spaces are geodesic, complete, and contractible.
	
	For a metric space $(X,d)$ with at least two points, define 
	$$\Delta(X)\coloneqq\{f \colon X\to \R\ |\ f(x)+f(y)\ge d(x,y) \text{ for } x,y\in X\}$$
	A function $f \colon X\to \R$ is called {\it extremal} if it is a minimal element in the partially ordered set $(\Delta(X),\le)$, where $g\le f$ if $g(x)\le f(x)$ for every $x\in X$. Define the set of extremal functions
	\begin{align*}
		E(X) \colon &=\{f\in \Delta(X)\ |\ \text{if }g\in \Delta(X)\text{ and }g\le f,\ \text{then }g=f\}\\
		&=\{f\in \Delta(X)\ |\ f(x)=\sup_{y\in X}(d(x,y)-f(y))\text{ for } x\in X\}
	\end{align*}
	(cf. \cite[Section 3]{lang2013injective}, \cite[Section 1]{dress1984trees}). In fact, if $f\in E(X)$ and $x\in X$ then 
	\begin{equation}\label{injhull}
		f(x)=\sup_{y\in X\setminus \{x\}}d(x,y)-f(y)
	\end{equation} 
	(see \cite[Section 1]{dress1984trees}), this fact will be useful later. Extremal functions are 1-Lipschitz. $E(X)$ is equipped with the metric $(f,g)\mapsto \|f-g\|_\infty$. The map  $e \colon X\to E(X)$, $x\mapsto d_x$, (where $d_x \colon X\to \R$, $y\mapsto d(x,y)$) is an isometric embedding and $(E(X),\|\cdot\|_\infty)$ is the injective hull of $X$.
	
	\bigskip

	\section{Antipodal relations and Polyhedral structure of $\MM(Z,\rho)$}\label{Antipodal relations and Polyhedral structure}
	
	In this section, we observe that for an antipodal space $(Z, \rho_0)$ of finite cardinality, we can derive combinatorial data in the form of the {\it antipodal relation}, which we will use to obtain the polyhedral structure of the Moebius space $\mathcal{M}(Z, \rho_0)$ as discussed in \Cref{1st Thm}.

	Unless otherwise specified, from now on, $(Z,\rho_0)$ will denote a finite antipodal space with at least four points. For brevity, we will often refer to $\mathcal{M}(Z,\rho_0)$ as simply $\mathcal{M}(Z)$ in this \Cref{Antipodal relations and Polyhedral structure} and the following \Cref{Injective metric space and MZ}. Whenever we write $\mathcal{M}(Z)$, it should be understood that there exists a fixed antipodal function $\rho_0$ for which $\mathcal{M}(Z) \coloneqq \mathcal{M}(Z,\rho_0)$. Additionally, due to the identification \eqref{embedding}, we will frequently identify antipodal functions $\rho\in\MM(Z)$, with the corresponding continuous functions $\tau_{\rho} = \log\frac{d\rho}{d\rho_0}$ on $Z$.

	\medskip

	\subsection{Antipodal relations}
	
	\noindent For an antipodal space $(Z,\rho_0)$ we define the following function from $\MM(Z)$ to the powerset $\mathcal{P}(Z\times Z\setminus\Delta)$ (where $\Delta\coloneqq\{\ (\xi,\xi)\ |\ \xi\in Z\ \}$):
	\begin{align*}
		\Phi \colon \MM(Z)&\to \mathcal{P}(Z\times Z\setminus\Delta)\\
		\rho&\mapsto R_\rho\coloneqq\{(\xi,\eta)\ |\ \rho(\xi,\eta)=1\}\subseteq (Z\times Z)\setminus\Delta.
	\end{align*}
	
	\noindent $R_\rho$ is a symmetric relation on $Z$ for each $\rho\in \MM(Z)$. We shall call a symmetric relation $R\subseteq(Z\times Z)\setminus\Delta$ {\it admissible} if for every $\xi \in Z$ there exists $\eta \in Z$ such that $(\xi,\eta)\in R$. A symmetric relation $R$ on $Z$ is called {\it antipodal} if either $R={\emptyset}$ or $R\in\ $Image($\Phi$). Let $\mathcal A_Z$ denote the set of all antipodal relations on $Z$, $\mathcal A_Z=$ Image$(\Phi)\cup\{\emptyset\}$. Also note that $\{\Phi^{-1}(R)\ |\ R \in \mathcal A_Z \}$ gives a partition of $\MM(Z)$.
	\medskip
	
	\subsection{Polyhedral structures}\label{polyhedral structure}
	
	\noindent Given $(Z, \rho_0)$, a finite antipodal space with $n$ points, we label the points as $\{1, 2, \dots, n\}$. For $i \neq j$, define $a_{ij} \coloneqq \log(\rho_0(i,j)^2)$.
	Define affine linear functions $l_{ij} \colon \R^n\to \R$ for $i\neq j$ by 
	\begin{equation*}
		l_{ij}(x)=x_i+x_j+a_{ij} \text{, for } \tau\in \R^n.
	\end{equation*}  Given an admissible relation $R$ we define the following subsets in $(\R^n,\|\cdot\|_\infty)\colon$ 
	\begin{align}
		C(R)&\coloneqq\{\tau\in \R^n \ |\  l_{ij}(\tau)=0\text{ for all }(i,j)\in R \text{ and } l_{ij}(\tau)\leq 0 \text{ for all } (i,j)\not\in \Delta\ \}\\
		C(R)^*&\coloneqq\{\tau\in \R^n\ |\  l_{ij}(\tau)=0\text{ for all }(i,j)\in R \text{ and } l_{ij}(\tau)<0 \text{ for all }(i,j)\not\in R\cup \Delta\}\label{set}
	\end{align} 
	We can immediately observe that :
	\begin{enumerate}[(i)]
		\item $C(R)$ is a polyhedron
		
		\item For admissible relations $R_1,R_4$ if $R_1\subseteq R_2$ then $C(R_2)\subseteq C(R_1)$.
		
		\item $C(R)^*$ is a subset of $C(R)$ and if $C(R)^*$ is non-empty then $C(R)=\overline{C(R)^*}$, the closure of $C(R)$ in the usual topology of $\R^n$. In particular, $C(R)^*$ is the relative interior of $C(R)$, that is, the interior of $C(R)$ in the subspace topology.
	\end{enumerate}
	Now, we observe a few more properties of the collection ${C(R)}$ in the following lemmas.
	\begin{sublemma}\label{cell}
		Let $(Z,\rho_0)$ be a finite antipodal space.
		\medskip
		
		(1) Given an admissible relation $R(\neq\emptyset)$ it is antipodal if and only if $C(R)^*\neq\emptyset$.\medskip
		
		(2) The isometric embedding $i_{\rho_0} \colon \MM(Z)\to (\R^n,\|\cdot\|_\infty)$ maps $\Phi^{-1}(R)$ onto $C(R)^*$ for any admissible relation $R\neq\emptyset$.\medskip
		
		(3) If $R$ is an antipodal relation, then $R\subset R_\rho$ for all $\rho\in E_{\rho_0}(C(R))$.\medskip 
		
		(4) For $R_1,R_2$ admissible relations we have $C(R_1)\cap C(R_2)=C(R_1\cup R_2)$.
		
	\end{sublemma}
	\begin{proof}
		Observe that given an admissible relation $R$, $R=R_\rho$ for some $\rho\in \MM(Z)$ if and only if 
		\begin{equation*}
			\rho(i,j)
			\begin{cases}
				=1 \quad &\text{if } \, (i,j)\in R\\
				<1 \quad &\text{if } \, (i,j)\not\in R
			\end{cases}
		\end{equation*}
		which by GMVT holds if and only if,
		\begin{equation*}
			\tau_\rho(i)+\tau_\rho(j)+\log(\rho_0(i,j)^2)
			\begin{cases}
				=0 \quad &\text{if } \, (i,j)\in R\\
				<0 \quad &\text{if } \, (i,j)\not\in R
			\end{cases}
		\end{equation*}
		that is if and only if $\tau_\rho\in C(R)^*$. From this we can conclude the claims of (1), (2) and (3).\medskip
		
		For (3), if $R_1$ and $R_2$ are admissible, then $R_1 \cup R_2$ is also admissible. Therefore, from the above, we have $C(R_1 \cup R_2) \subseteq C(R_1) \cap C(R_2)$. If $C(R_1) \cap C(R_2) = \emptyset$, the result follows immediately. Otherwise, suppose $C(R_1) \cap C(R_2) \neq \emptyset$. Then $\tau \in C(R_1) \cap C(R_2)$ implies that $l_{ij}(\tau) = 0$ for all $(i, j)$ in both $R_1$ and $R_2$, and $l_{ij}(\tau) \leq 0$ for $i \neq j$. Hence, $\tau \in C(R_1 \cup R_2)$. Thus, we conclude that $C(R_1) \cap C(R_2) = C(R_1 \cup R_2)$.
	\end{proof}
	
	As a consequence of the above \Cref{cell} we have 
	
	\begin{equation*}
		\MM(Z,\rho_0)=\bigcup_{R \in \mathcal A_Z} \Phi^{-1}(R)=\bigcup_{R \in \mathcal A_Z}E_{\rho_0}(C(R)^*)
	\end{equation*}
	(where $E_{\rho_0} \colon (\R^n,\|\cdot\|_\infty)\to \mathcal{UM}(Z,\rho_0)$ is the inverse of the isometry $i_{\rho_0}$ from $\mathcal{UM}(Z,\rho_0)$ to $(\R^n,\|\cdot\|_\infty)$).
	Note that since $\MM(Z)$ is closed in $\mathcal{UM}(Z,\rho_0)$, 
	$E_{\rho_0}(C(R))=E_{\rho_0}(\overline{C(R)^*})=\overline{E_{\rho_0}(C(R)^*)}\subseteq\MM(Z)$ for any $R\in \mathcal A_Z$.  As a consequence, we have 
	\begin{equation*}
		\MM(Z,\rho_0)=\bigcup_{R \in \mathcal A_Z}E_{\rho_0}(C(R)).
	\end{equation*} Thus, the isometric embedding $i_{\rho_0}$, isometrically maps $\MM(Z,\rho_0)$ to 
	\begin{equation}\label{union}
		\mathcal{P}\coloneqq\bigcup_{R \in \mathcal A_Z}C(R)\subset(\R^n,\|\cdot\|_\infty)
	\end{equation}
	which is a finite union of polyhedrons in $\R^n$ (note that $\mathcal A_Z$ is finite since $Z$ is finite).
	\medskip
	
	Given an admissible relation $R$, we define a graph $\Gamma_R$ with vertex set $Z$, where two distinct vertices $i, j \in Z$ (where $i \neq j$) are connected by an edge if and only if $(i, j), (j, i) \in R$. For such graphs, there are two possible cases:
	
	\begin{enumerate}[{\it C{a}se 1:}]
		\item Either there is a vertex say $i_0\in Z$ such that all other vertices are connected to $i_0$ by an edge and there are no other edges or
		\item there exist four distinct vertices $i,j,k,l\in Z$ such that $i,j$ are connected by an edge and $k,l$ are connected by an edge.
	\end{enumerate} 
	If the graph $\Gamma_R$ is as in Case 1, we say the relation $R$ is of {\it Type 1} while in Case 2, we say $R$ is of {\it Type 2}.
	\begin{figure}[H]
		\centering
		\begin{tikzpicture}[scale=0.5,line cap=round,line join=round,>=triangle 45,x=1cm,y=1cm]
			\clip(-7,-3) rectangle (7,3);
			\draw [line width=0.4pt] (-5,2)-- (-3,-2);
			\draw [line width=0.4pt] (-5,2)-- (-6,0);
			\draw [line width=0.4pt] (-5,2)-- (-5,-2);
			\draw [line width=0.4pt] (-5,2)-- (-2,0);
			\draw [line width=0.4pt] (-5,2)-- (-3,2);
			\draw [line width=0.4pt] (3,2)-- (2,0);
			\draw [line width=0.4pt] (3,2)-- (6,0);
			\draw [line width=0.4pt] (3,2)-- (5,2);
			\draw [line width=0.4pt] (3,-2)-- (5,-2);
			\draw [line width=0.4pt] (-4,2)-- (-5,2);
			\begin{scriptsize}
				\draw [fill=black] (-6,0) circle (2.5pt);
				\draw[color=black] (-6.397,0.4457) node {$6$};
				\draw [fill=black] (-5,2) circle (2.5pt);
				\draw[color=black] (-5.2838,2.4301) node {$1$};
				\draw [fill=black] (-3,2) circle (2.5pt);
				\draw[color=black] (-2.8396,2.4543) node {$2$};
				\draw [fill=black] (-2,0) circle (2.5pt);
				\draw[color=black] (-1.8232,0.4457) node {$3$};
				\draw [fill=black] (-5,-2) circle (2.5pt);
				\draw[color=black] (-5.4048,-1.5871) node {$5$};
				\draw [fill=black] (-3,-2) circle (2.5pt);
				\draw[color=black] (-2.8396,-1.5629) node {$4$};
				\draw [fill=black] (3,2) circle (2.5pt);
				\draw[color=black] (3.162,2.4543) node {$1$};
				\draw [fill=black] (5,2) circle (2.5pt);
				\draw[color=black] (5.1706,2.4543) node {$2$};
				\draw [fill=black] (6,0) circle (2.5pt);
				\draw[color=black] (6.1628,0.4457) node {$3$};
				\draw [fill=black] (3,-2) circle (2.5pt);
				\draw[color=black] (2.5086,-1.9501) node {$5$};
				\draw [fill=black] (5,-2) circle (2.5pt);
				\draw[color=black] (5.3884,-1.9743) node {$4$};
				\draw [fill=black] (2,0) circle (2.5pt);
				\draw[color=black] (1.6132,-0.0141) node {$6$};
			\end{scriptsize}
		\end{tikzpicture}
		\caption{Examples of graphs for Type 1 and Type 2 relations for $Z=\{1,2,\cdots,6\}$}
	\end{figure}
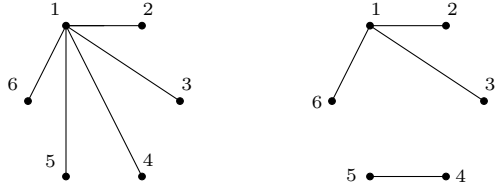
	
	\medskip
	
	\noindent If $R$ is of Type 1, that is, in $\Gamma_R$, vertex $i_0$ (let us say $i_0=1$ for ease of presentation) is connected to every other vertex, then for $\tau\in C(R)^*$
	\begin{equation*}
		\tau(j)=-\tau(1)-a_{1j} \text{ for all } j \neq 1,
	\end{equation*} 
	and
	\begin{equation*}
		\tau(i)+\tau(j)+a_{ij}<0 \ \ \ \ \text{ for all } i\neq j ,
	\end{equation*}
	Hence, it follows that
	\begin{equation*}
		\frac{-a_{1i} - a_{1j} + a_{ij}}{2} < \tau(1) \quad \text{for all } i \neq j \neq 1
	\end{equation*}
	Therefore, we have
	\begin{equation}\label{ray}
		C(R)^* = \biggl\{ \ (t, -t - a_{12}, \dots, -t - a_{1n}) \ \big| \ t \ge \max_{i \neq j \neq 1} \frac{-a_{1i} - a_{1j} + a_{ij}}{2} \ \biggr\} \neq \emptyset,
	\end{equation}
	and $R$ is an antipodal relation. Moreover, the closed polyhedral cell $C(R)$ is a geodesic ray with the end-point (at infinity) $1 \in \partial \MM(Z)$. Note that there are exactly $n$ Type 1 admissible relations. For each fixed $i \in Z$, we denote the Type 1 relation by $R_0^i$ as follows:
	\begin{equation*}
		R_0^i \coloneqq \{(i, j), (j, i) \mid j \in Z, j \neq i\}.
	\end{equation*}
	The above calculation shows that $C(R_0^i)^*$ is non-empty; hence, $R_0^i$ is always an antipodal relation. The cell $C(R_0^i)$ is the geodesic ray with end-point $i \in \partial \MM(Z) = Z$.
	\medskip
	
	\noindent Otherwise, for $R$ of Type 2, if $\rho\in \Phi^{-1}(R)\neq \emptyset$ then 
	\begin{equation*}
		\rho(i,j)=\exp(-(i|j)_\rho)=1=\exp(-(k|l)_\rho)=\rho(k,l)
	\end{equation*}
	for some four distinct points $i,j,k,l\in Z$ (by \Cref{gromov ip}) which implies $(i|j)_\rho=(k|l)_\rho=0$. Thus, $\rho$ lies at the intersection of a bi-infinite geodesic joining $i$ to $j$ and another bi-infinite geodesic joining $k$ to $l$. Then by \Cref{bddlema}, $\Phi^{-1}(R)$ is bounded and so is $C(R)$ as well. Therefore, if $R$ is of Type 2, the polyhedral cell $C(R)$ is a polytope.
	
	\begin{subrmk}\label{structure}
		It follows that for a finite antipodal space $(Z, \rho_0)$ consisting of $n \geq 4$ points, $\MM(Z)$ is the union of $n$ distinct geodesic rays (which are the only polyhedral cells $C(R)$ that are unbounded) and a compact part composed of the union of a finite number of polytopes in $(\R^n, \|\cdot\|_\infty)$.
	\end{subrmk}\medskip
	
	Now, we observe some properties for the faces of the polyhedrons $C(R)$, where $R$ is an antipodal relation.
	\begin{sublemma}\label{face}
		For a finite antipodal space $(Z,\rho_0)$\medskip
		
		(1)  If $R_1,R_2$ are antipodal relations and $\ \emptyset\neq R_1\subseteq R_2$ then $C(R_2)$ is a face of $C(R_1)$.\medskip
		
		(2) If $R$ is an antipodal relation and $F\subset C(R)$ is a proper non-empty face of $C(R)$ then $F=C(R')$ for some antipodal relation $R'$ with $R\subset R'$.
	\end{sublemma}
	\begin{proof}[{\it Proof of (1)}]
		If $\#R_2-\#R_1=0$ then $R_1=R_2$, nothing to prove.
		If $\#R_2-\#R_1=2$ then there exist $i\neq j$ in $Z$ such that $R_2\setminus R_1={(i,j),(j,i)}$. By definition $C(R_1)\subset\{\tau\in \R^n\ |\ l_{ij}(\tau)=0\}$ and $C(R_2)=C(R_1)\cap\{\tau\in \R^n\ |\ l_{ij}(\tau)=0\}\neq \emptyset$. So $C(R_2)$ is a face of $C(R_1)$.
		
		Suppose $\#R_2-\#R_1=2k>2$ then we have $k$ antipodal relations $R^i$, $i=1,2,\cdots,k$ such that $R_1\subset R^i\subset R_2$, and $\#R^i-\#R_1=2$ with $\cup_{i=1}^kR^i=R_2$. By above argument each $C(R^i)$ is a face of $C(R_1)$, and by \Cref{cell}(4) we get $\cap_{i=1}^k C(R^i)=C(R_2)$. We know that the intersection of faces of the polyhedral cells is a face as well, so $C(R_2)$ is a face of $C(R_1)$.
		\medskip
		
		\noindent {\it Proof of (2)}. Let $H$ be a hyperspace. Let $H^+$ and $H^-$ denote the two open half spaces bounded by $H$. Without loss of generality let $C(R)\subset \overline{H^+}$ and the face $F=C(R)\cap H$. Consider the collection 
		\begin{equation*}
			\mathcal R=\{\ R_\rho\ |\ \rho\in \MM(Z,\rho_0) \text{ and }\tau_\rho\in F\ \}=\{\ R_\rho\ |\ \rho\in E_{\rho_0}(F)\ \}.
		\end{equation*} Since $F \subset C(R)$, observe that $R \subset R_\rho$ for each $\rho \in E_{\rho_0}(F)$ (see \Cref{cell}(3)). Consider the partial order by inclusion on $\mathcal{R}$. As $\mathcal{R}$ is a finite collection, it has a minimal element, say $R_{\rho_1}$. We will show that this minimal element is unique and $F = C(R_{\rho_1})$.
		
		Let there be another minimal element, say $R_{\rho_2}$. Now we know that $R\subset R_{\rho_1}\cap R_{\rho_2}$, so $R_{\rho_1}\cap R_{\rho_2}$ is an admissible relation. Observe that if we take $\rho=\sqrt{(\rho_1\cdot\rho_2)}$ then $\tau_\rho=\frac{1}{2}(\tau_{\rho_1}+\tau_{\rho_2})\in F$ and $R_\rho=R_{\rho_1}\cap R_{\rho_2}$. Now by minimality we must have $R_\rho=R_{\rho_1}=R_{\rho_2}$. So the minimal element $R_{\rho_1}$ in $\mathcal R$ is unique. This shows given any $\rho\in E_{\rho_0}(F)$ we have $R_{\rho_1}\subseteq R_{\rho}$ and hence $\tau_\rho\in C(R_\rho)\subseteq C(R_{\rho_1})$, implies $F\subseteq C(R_{\rho_1})$.
		
		\noindent \underline{Claim:} $C(R_{\rho_1})^*\subseteq F$ and hence $ C(R_{\rho_1})= F$, which completes our proof.
		
		Suppose the claim is not true, then there exists $\tau'\in C(R_{\rho_1})^*\setminus F\subset C(R)$. However, $\tau'\not\in H$, so $\tau'\in H^+$. Also, $\tau'\neq\tau_{\rho_1}\in F$ but both are in $C(R_{\rho_1})^*$. Note, $C(R_{\rho_1})^*$ being open subset of the affine subspace $A(R_{\rho_1})$, where
		\begin{equation*}
			A(R_{\rho_1})\coloneqq\{\ \tau \in \R^n\ |\ l_{ij}(\tau)=0\ \text{for }(i,j)\in R_{\rho_1}\ \}\subset \R^n,
		\end{equation*} we can extend the line segment $[\tau',\tau_{\rho_1}]$ contained in $C(R_{\rho_1})^*$ further to get $\tau''\in C(R_{\rho_1})^*$ such that $\tau_{\rho_1}=t\cdot\tau'+(1-t)\cdot\tau''$ for some $t\in (0,1)$. But $\tau'\in H^+$ and $\tau_{\rho_1}\in H$ implies $\tau''\in H^-$ hence $\tau''\not\in C(R_{\rho_1})$ which implies $\tau''\not\in C(R_{\rho_1})^*$. This is a contradiction. So we must have $C(R_{\rho_1})^*\subseteq F$.
	\end{proof}
	\noindent Now we are in a position to prove \Cref{1st Thm}
	
	\begin{proof}[{\bf Proof of \Cref{1st Thm}}] From \Cref{antipodal} we know for $(Z,\rho_0)$ finite antipodal space, $\MM(Z)$ is isometrically embedded in $(\R^n,\|\cdot\|_\infty)$. Consider the collection of polyhedral cells $\mathcal C\coloneqq\{C(R)\ |\ R\in \mathcal A_Z\}.$ Note, $\emptyset\in\mathcal C$ as empty relation is in $\mathcal A_Z$. By \Cref{face}(2) we have every face $F$ of a cell $C(R)$ is equal to $C(R')$ for some $R'\in \mathcal A_Z$. Consider any two cells $C(R_1),C(R_2)\in \mathcal C$ either $C(R_1)\cap C(R_2)=\emptyset\in \mathcal C$ or if $\emptyset \neq C(R_1)\cap C(R_2)$ then by \Cref{cell} $C(R_1)\cap C(R_2)=C(R_1\cup R_2)$ then $R_1\cup R_2\in \mathcal A_Z$ and hence $C(R_1\cup R_2)\in \mathcal C$. So the intersection of two cells $C(R_1), C(R_2)$ is in $\mathcal C$, and by \Cref{face}(1), it is the face of both the cells. Thus, $\mathcal{C}$ forms a polyhedral complex, where the underlying set $|\mathcal{C}|$ is equal to $\mathcal{P}$ (as defined in \eqref{union}). Consequently, $\MM(Z,\rho_0)$ is isometric to a polyhedral complex embedded in $(\mathbb{R}^n, \|\cdot\|_\infty)$. Finally, we have the rest from \Cref{structure}.
	\end{proof}
	
	\begin{figure}
		\centering
		\includegraphics[width=1\linewidth]{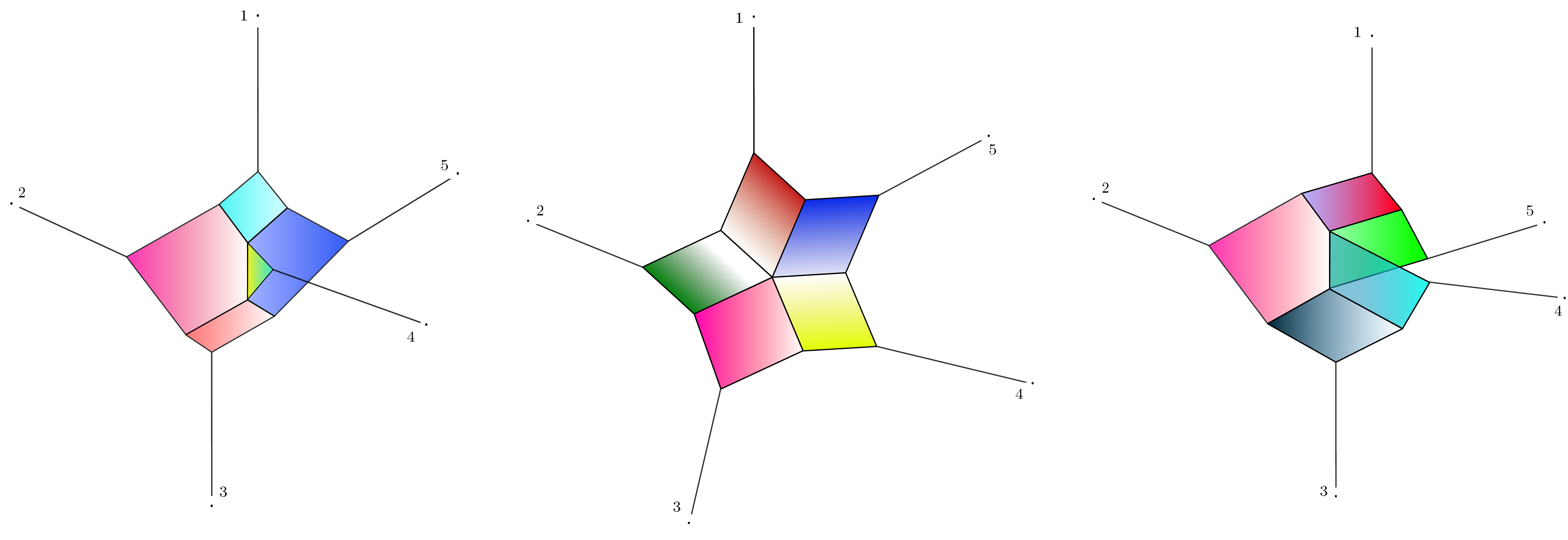}
		\caption{Example of possible polyhedral structures of $\MM(Z,\rho)$ for $Z=\{1,2,3,4,5\}$}
		\label{fig:drawing}
	\end{figure}
	
	\medskip
	
	\subsection{Dimension of the polyhedral cells}
	
	We now determine the dimensions of the non-empty cells $C(R)$, where $R\in \mathcal A_Z\setminus\{\emptyset\}$. In this case $C(R)^*$ is non-empty and $C(R)=\overline{C(R)^*}$. Note that 
	$C(R)^*$ the relative interior of $C(R)$ and is an open subset of the affine subspace $A(R)$, where
	\begin{equation}\label{affine subspace}
		A(R)\coloneqq\{\ \tau \in \R^n\ |\ l_{ij}(\tau)=0\ \text{for }(i,j)\in R\ \}\subset \R^n
	\end{equation} equipped with the subspace topology inherited from $\R^n$. Thus, the dimension of the polyhedral cell $C(R)$ is equal to the dimension of $C(R)^*$, which is equal to the dimension of $A(R)$. Observe that the dimension of $A(R)$ is less than or equal to the number of connected components in the graph $\Gamma_R$, and $R$ being an admissible relation, the number of connected components is at most $n/2$ (where $n=\#Z$). 
	
	The discussion in this paragraph is similar to that of \cite[Section 4, p. 312-313]{lang2013injective}. For $R\neq\emptyset$, antipodal relation, in the graph $\Gamma_R$ a {\it path} of length $\ell>0$ is a sequence of $\ell+1$ vertices such that the two consecutive vertices in the sequence are distinct and are connected by an edge. A {\it cycle} is a path with the last vertex in the sequence equal to the first one. Cycles are of length at least $2$, as there are no self-loops.
	Fix $\tau_0\in A(R)$. Suppose $(i,j)\in R$ is an edge in $\Gamma_R$ then for all $\tau \in A(R)$ we have $l_{ij}(\tau)=0$ and hence $\tau_0(i)-\tau(i)=-(\tau_0(j)-\tau(j))$. If $i$ and $j$ connected by a path of length $l$ in $\Gamma_R$ then 
	\begin{equation}\label{relation}
		\tau_0(i)-\tau(i)=(-1)^{\ell}(\tau_0(j)-\tau(j))
	\end{equation}
	Consider the connected component of $i\in Z$ in the graph $\Gamma_R$
	$$[i]=\{\ j\in Z\ |\ j \text{ is connected to } i \text{ by a path }\}\cup \{i\}$$
	Now if the component $[i]$ contains a cycle of odd length then $\tau|_{[i]}=\tau_0|_{[i]}$, by taking $i=j$ in \eqref{relation}. We shall call $[i]$ an {\it odd-component} if it contains a loop of odd length. Otherwise, we call it an {\it even-component}, where the connected component does not contain any loop of odd length. Given an even-component $[i]$ the sub-graph $[i]$ in $\Gamma_R$ is bipartite, i.e. it admits a unique partition 
	$$[i]=[i]_0\cup[i]_1$$ 
	such that $i\in [i]_0$ and every element in $[i]_1$ is connected to an element in $[i]_0$ by an edge. In particular take 
	\begin{align*}
		[i]_1&\coloneqq\{\ j\in [i]\ |\ j \text{ is connected to } i \text{ by a path of odd length }\}\\
		[i]_0&\coloneqq[i]\setminus[i]_1
	\end{align*}
	By \eqref{relation} for $\tau\in A(R)$ we have 
	\begin{align*}
		\tau(j)-\tau_0(j)=\tau(i)-\tau_0(i)\quad &\hbox{ if }\quad j\in[i]_0\\
		\tau(j)-\tau_0(j)=\tau_0(i)-\tau(i)\quad &\hbox{ if }\quad j\in[i]_1
	\end{align*} 
	
	Suppose $[i_1],\cdots,[i_K]$ are all the distinct even components in $\Gamma_R$. Then from above $\tau\in A(R)$ if and only if $\tau\in \R^n$ and
	
	\medskip 
	\noindent (1)  on odd components $\tau=\tau_0$
	
	\noindent (2) on even components $[i_k]$ for $k=1,\cdots,K$
	\begin{align*}
		\tau(j)=\tau_0(j)+t_k\quad &\hbox{ if }\quad j\in[i_k]_0\\
		\tau(j)=\tau_0(j)-t_k\quad &\hbox{ if }\quad j\in[i_k]_1
	\end{align*} 
	for some $(t_1,\cdots,t_K)\in \R^K$. In this case,  $\dim(A(R))=K$. Thus, it is clear that the dimension of $A(R)$ is precisely equal to the number of even-components in $\Gamma_R$.
	\medskip
	
	Now, we shall discuss the tangent spaces of $\MM(Z)$. Let $\rho\in\MM(Z)$ then $R_\rho$ is an antipodal relation. From  \Cref{tangent} the tangent space of $\MM(Z)$ at $\rho$, $T_\rho\MM(Z)=\OO_\rho(Z)\subset\R^n$. Let $\tau\in\R^n$, then $\tau\in\OO_\rho(Z)$, if and only if $\tau\in \R^n$ and $\tau(i)+\tau(j)=0$ whenever $(i,j)\in R_\rho$, if and only if 
	$l_{ij}(\tau)=a_{ij}$ whenever $(i,j)\in R_\rho$. We know that $\tau_\rho=\log\frac{d\rho}{d\rho_0}\in A(R_\rho)$. Suppose $\tau\in A(R_\rho)$, then whenever $(i,j)\in R_\rho$, 
	\begin{align*}
		&l_{ij}(\tau)=l_{ij}(\tau_\rho)=0\\
		{\iff}&\tau(i)+\tau(j)+a_{ij}=\tau_\rho(i)+\tau_\rho(j)+a_{ij}=0\\
		{\iff}&	(\tau(i)-\tau_\rho(i))+(\tau(j)-\tau_\rho(j))=0\\
		{\iff}& l_{ij}(\tau-\tau_\rho)=a_{ij}
	\end{align*} 
	So we conclude that $$\OO_\rho(Z)=A(R_\rho)-\tau_\rho,$$ moreover
	$$\dim (T_\rho\MM(Z))=\dim (\OO_\rho(Z))=\dim (A(R_\rho)).$$
	\begin{subrmk}
		Given any antipodal relation $R$ by \eqref{set} $R_\rho=R$ for all $\rho\in C(R)^*$. So we have the dimension of cell $C(R)$ is equal to the dimension of the tangent space $T_\rho\MM(Z)$, for every $\rho\in C(R)^*$.
	\end{subrmk}
	
	\bigskip 
	
	\section{Injectivity of maximal Gromov hyperbolic spaces with finite boundary}\label{Injective metric space and MZ}
	\noindent Let $(Z,\rho_0)$ be a quasi-metric antipodal space (not necessarily finite). For a fixed $r>0$ consider the metric sphere $S\coloneqq S_\MM(\rho_0,r)\subset \MM(Z)$. For $\tau \colon S\to \R$ a continuous function we define discrepancy of $\tau$ at finite distance, motivated by discrepancy function defined in  \cite[Definition 3.1]{biswas2024quasi}, as $D^r_{\rho_0}(\tau) \colon S\to \R$ by
	\begin{equation}
		D^r_{\rho_0}(\tau)(x)\coloneqq\sup_{y\in S\setminus\{x\}}\tau(x)+\tau(y)-2(x|y)_{\rho_0}
	\end{equation}
	
	Now observe that for $\tau\in C(S)\ $, we have $\ (r-\tau)\in E(S)$, the injective hull of $S$, if and only if $D^r_{\rho_0}(\tau)\equiv0$. This is because,
	\begin{equation*}
		\begin{split}
			D^r_{\rho_0}(\tau)(x)=&\sup_{y\in S\setminus\{x\}}\tau(x)+\tau(y)-2(x|y)_{\rho_0}=0\ , \quad \hbox{for all } x\in S\\
			\Longleftrightarrow& \sup_{y\in S\setminus\{x\}}\tau(x)+\tau(y)-r-r+d_\MM(x,y)=0\ ,\quad \hbox{for all } x\in S\\
			\Longleftrightarrow& \sup_{y\in S\setminus\{x\}}d_\MM(x,y)-(r-\tau(y))=(r-\tau(x))\ ,\quad \hbox{for all } x\in S
		\end{split}
	\end{equation*} 
	which is equivalent to $\ (r-\tau)\in E(S)$, from equation (\ref{injhull}).
	Also, $f\in E(S)$ if and only if $D^r_{\rho_0}(r-f)\equiv0$. Hence, we can write the following lemma.
	\begin{lemma}\label{ballembedding}
		$f\in E(S)$ if and only if $f \colon S\to \R$ is a continuous function  and $D^r_{\rho_0}(r-f)\equiv0$.
	\end{lemma}
	
	\noindent Given a point $\rho\in B_\MM(\rho_0,r)$\footnote{Here, $B_\MM(\rho_0,r)= \{\ \rho\in \MM(Z)\ |\ d_\MM(\rho,\rho_0)\le r\ \}$ closed ball in $\MM(Z).$}  we have the distance function $d_{\rho} \colon S\to \R\ (x\mapsto d_\MM(x,\rho))$ restricted to $S$
	\begin{lemma}
		The map 
		\begin{align*}
			\mathcal E \colon B_\MM(\rho_0,r)&\to E(S)\\
			\rho&\mapsto d_{\rho}
		\end{align*}
		is an isometric embedding and $D^r_{\rho_0}(r-d_\rho)\equiv0$ for $\rho\in B_\MM(\rho_0,r)$.
	\end{lemma}
	\begin{proof}
		Let $\rho\in B_\MM(\rho_0,r)$. We need to show $d_{\rho}(x)=\sup_{y\in S}d_\MM(x,y)-d_{\rho}(y)$ for all $x\in S$. Let $x\in S$ if $\rho=x$, then $d_\rho(y)=0=d_\MM(x,y)-d_\rho(y)$ for all $y\in S$, so we are done. Otherwise, if $\rho\neq x$ then by geodesic completeness of $\MM(Z)$ we can find a bi-infinite geodesic $\gamma$ containing $x$ and $\rho$, which intersects the metric sphere $S$ at point $y_0\neq x$ such that $\rho$ is on the segment $[xy_0]$ of $\gamma$. Then $d_{\rho}(x)= d_\MM(x,y_0)-d_{\rho}(y_0)=\sup_{y\in S}d_\MM(x,y)-d_{\rho}(y)$. Therefore, $d_\rho\in E(S)$.
		
		Now let $\rho_1,\rho_2\in  B_\MM(\rho_0,r)$. Now in $\MM(Z)$ there is a geodesic segment $[\rho_1,\rho_2]$. Again, by geodesic completeness of $\MM(Z)$, we can find a bi-infinite geodesic $\gamma'$, which is an extension of $[\rho_1,\rho_2]$. $\gamma'$ intersect the metric sphere $S$ at a point $x$  such that $\rho_2$ is on the segment $[\rho_1,x]$ of $\gamma'$. So we have  $d_\MM(\rho_1,\rho_2)=d_\MM(\rho_1,x)-d_\MM(\rho_2,x)$. Therefore, $\|d_{\rho_1}-d_{\rho_2}\|_\infty=d_\MM(\rho_1,\rho_2)$
		
		\noindent Hence $\rho\mapsto d_{\rho}$ is an isometric embedding, and we also have $D^r_{\rho_0}(r-d_\rho)\equiv0$.
	\end{proof}
	\medskip
	
	Let $(Z, \rho_0)$ be a finite antipodal space with $n \geq 4$ points. For $r > 0$, define $S_r \coloneqq S(\rho_0, r) \subseteq \MM(Z)$, the metric sphere centered at $\rho_0$. By \Cref{1st Thm}, we know that for sufficiently large $r > 0$, $S_r$ is a finite subset with cardinality $n$, containing one point from each geodesic ray extending to $\partial \MM(Z)$. Thus, we define
	$$
	S_r = \{x^r_i \mid i \in Z\},
	$$
	where $x^r_i$ denotes the unique point at the intersection of $S_r$ and the geodesic ray that extends to $i \in \partial \MM(Z) = Z$. We know $\MM(Z)$ is boundary continuous, so we have the Gromov product 
	$$(i|j)_{\rho_0}\coloneqq\lim_{x\to i,y\to j}(x|y)_{\rho_0}$$
	\noindent for any point $i,j\in\partial\MM(Z)=Z$. Observe that any geodesic ray $[\rho_0,i)$, where $i\in \partial\MM(Z)=Z$, must pass through $x^r_i$. So the Gromov product $(i|j)_{\rho_0}$, which we get as a limit, is attained at $(x^r_i|x^r_j)_{\rho_0}$.  Form \Cref{gromov ip} we have
	\begin{equation}\label{ip}
		\rho_0(i,j)=\exp(-(x^r_i|x^r_j)_{\rho_0}).
	\end{equation}
	\noindent By  \cite[Proposition 6.1]{biswas2024quasi} we know if $\rho\in \MM(Z)$ and $i\in \partial\MM(Z)=Z$
	$$\tau_\rho(i)=\log\frac{d\rho}{d\rho_0}(i)=\lim_{x\to i}d_\MM(x,\rho_0)-d_\MM(x,\rho)$$
	(this is true in general $\MM(Z)$). Again observe that, for $\rho\in B_\MM(\rho_0,r)$ any geodesic ray $[\rho,i)$, where $i\in \partial \MM(Z)=Z$, must pass through $x^r_i$. Then, by the same justification as above, we have 
	
	\begin{equation}\label{busemann}
		\tau_\rho(i)=\log\frac{d\rho}{d\rho_0}(i)=d(x_i,\rho_0)-d(x_i,\rho)=r-d_\rho(x^r_i)
	\end{equation} 
	
	Fix $\tilde{r}>0$ such that $\MM(Z)\setminus B_\MM(\rho_0,\tilde{r})$ is union of $n$ many rays, the metric sphere $S_{\tilde{r}}$ is a finite metric space, and equations \eqref{ip} and \eqref{busemann} holds. For ease of notation denote $x_i=x^{\tilde{r}}_i$. Given a continuous function $\tau$ on $Z$ we can view it as a function on $S_{\tilde{r}}$ and vice-versa by the identification of $Z$ and $S_{\tilde{r}}$,  $(i\leftrightarrow x_i)$ i.e. define $\tau(i)=\tau(x_i)$. The following \Cref{inter} shows some properties of extremal functions on $S_{\tilde{r}}$, and the 
	
	\begin{lemma}\label{inter}
		Let $f\in E(S_{\tilde{r}})$, an extremal function on $S_{\tilde{r}}$. (Recall the definition of discrepancy  \cite[Definition 3.1]{biswas2024quasi}) Then $f$ seen as a function on $Z$ by the above identification we have $D_{\rho_0}(\tilde{r}-f)\equiv0$ and $E_{\rho_0}(\tilde{r}-f)\in B_\MM(\rho_0,\tilde{r})$.
	\end{lemma}
	\begin{proof}
		$f\in E(S_{\tilde{r}})$ if and only if $D^{\tilde{r}}_{\rho_0}(\tilde{r}-f)\equiv 0$.  For $i\in Z$ from equation \eqref{ip} 	
		\begin{align*}
			D_{\rho_0}(\tilde{r}-f)(i)=& \sup_{j\in Z\setminus\{i\}} \tilde{r}-f(i)+\tilde{r}-f(j)+\log \rho_0(i,j)^2\\
			=&\sup_{j\in Z\setminus\{i\}}\tilde{r}-f(i)+\tilde{r}-f(j)-2(i|j)_{\rho_0}\\
			=& \sup_{x_j\in S\setminus\{x_i\}} \tilde{r}-f(x_i)+\tilde{r}-f(x_j)-2(x_i|x_j)_{\rho_0}\\
			=&D^{\tilde{r}}_{\rho_0}(\tilde{r}-f)(x_i)=0
		\end{align*}
		So $D_{\rho_0}(\tilde{r}-f)\equiv 0$. For $f\in E(S)$ we know
		$$0\le f(x_i) \le \diam(S_{\tilde{r}})\le2\tilde{r}.$$
		Then seen as a function on $Z$ we have
		$$0\le f(i) \le2\tilde{r}$$ 
		which implies $\|\tilde{r}-f\|_\infty\le \tilde{r}$ in $(C(Z),\|\cdot\|_\infty)$. Hence, $E_{\rho_0}(\tilde{r}-f)\in B_\MM(\rho_0,\tilde{r})$.
	\end{proof}
	\noindent
	\begin{rmk} \label{exten}
		For $r\ge \tilde{r}$, the metric sphere $S_r$ are finite metric spaces, and equations \eqref{ip} and \eqref{busemann} are satisfied. The above \Cref{inter} can be proved similarly for all $r\ge \tilde{r}$.
	\end{rmk}
	Now, we shall prove the following proposition.
	\begin{prop}\label{Main Prop}
		Given a finite antipodal space $(Z,\rho_0)$,
		every closed ball in $\MM(Z,\rho_0)$ is an injective metric space. 
	\end{prop}
	\begin{proof}
		We shall prove that for all $r\ge \tilde{r}$ each closed ball $B_\MM(\rho_0,r)$ is the injective hull of the metric sphere $S_r=S_\MM(\rho_0,r)$. It will follow every closed ball in $\MM(Z)$ is injective. It is enough to prove the isometric embedding $\mathcal E \colon B_\MM(\rho_0,r)\to E(S_r)$ defined in \Cref{ballembedding} is surjective in this case. For $\rho\in B_\MM(\rho_0,r)$, by identification of $S_r$ and $Z$ $(i\leftrightarrow x_i)$ taking $\tau_\rho(x^r_i)\coloneqq\tau_\rho(i)$ we have  
		\begin{align*}
			D^r_{\rho_0}(\tau_\rho)(x^r_i)=&\sup_{x_j\in S\setminus\{x^r_i\}}\tau_\rho(x^r_i)+\tau_\rho(x^r_j)-2(x^r_i|x^r_j)_{\rho_0}\\
			=& \sup_{j\in Z\setminus\{i\}} \tau_\rho(i)+\tau_\rho(j)-2(i|j)_{\rho_0}\\
			=& \sup_{j\in Z\setminus\{i\}} \tau_\rho(i)+\tau_\rho(j)+\log \rho_0(i,j)^2\\
			=&D_{\rho_0}\tau_\rho(i)=0
		\end{align*}
		So discrepancy at finite distance $D^r_{\rho_0}(\tau_\rho)\equiv 0$ and therefore $r-\tau_\rho\in E(S_r)$. Also, from equation \eqref{busemann} $\tau_\rho(x^r_i)=\tau(i)=r-d_\rho(x_i)$ for all $x_i\in S$.
		Let $f\in E(S_r)$ then by \Cref{inter} along with \Cref{exten} we know that $E_{\rho_0}(r-f)\in B_\MM(\rho_0,r)$ for $r\ge \tilde{r}$. Take $\rho=E_{\rho_0}(r-f)$ then for $x_i\in S$ 
		$$r-f(x_i)=r-f(i)=\tau_\rho(i)=\tau_\rho(x_i)=r-d_\rho(x_i)$$
		which implies $f=d_\rho$ and hence $\mathcal E(\rho)=f$. Therefore, $\mathcal E$ is surjective.
		So for $r\ge \tilde{r}$, $B_\MM(\rho_0,r)$ is injective hull of the metric sphere $S_r$.
	\end{proof}

	Observe the following fact about injective metric spaces from \cite[p. 30]{petrunin2023pure}:
	\medskip
	
	\noindent \underline{{\bf Fact:}} If $X$ is a proper metric space and finitely hyperconvex then $X$ is hyperconvex, hence injective. 
	\medskip
	
	\noindent A space $X$ is said to be finitely hyperconvex if the condition (\ref{hyp}) in \Cref{hyperconvex thm} holds for a finite family of balls. Proof of the fact follows from the Cantor intersection property. We use this fact to prove the injectivity of maximal Gromov hyperbolic spaces.
	
	\begin{theorem}
		Any maximal Gromov hyperbolic space with finite boundary points is an injective metric space. That is, maximal Gromov hyperbolic spaces are hyperconvex.
	\end{theorem}
	
	\begin{proof}
		We know that for finite antipodal space $(Z,\rho_0)$, every closed ball in $\MM(Z,\rho_0)$ is injective by \Cref{Main Prop}, hence finitely hyperconvex. We also know $\MM(Z,\rho_0)$ is a proper metric space (see \cite[Lemma 2.7]{biswas2024quasi}). Then, by the above fact, we have $\MM(Z,\rho_0)$ as an injective metric space. 
	\end{proof}
	
	\bigskip
	
	\section{Teichmuller spaces of $\MM(Z, \rho)$ for finite $Z$}\label{Teich}
	\noindent In this section, we aim to explore the space of deformations of the maximal Gromov hyperbolic space with a fixed finite number of boundary points. For this section, unless otherwise mentioned, $Z$ is a finite set with $m\ (\ge 4)$ elements.
	
	\subsection{Big-Teichmuller Space}
	Let $Z$ be a finite set, $\ANZ$ denote the set of all antipodal functions on $Z$. We say two antipodal functions $\rho_1,\rho_2\in \ANZ$ are Moebius equivalent if they have the same cross-ratios (see \Cref{same cross-ratio}), denoted by $\rho_1\sim_M \rho_2$ and denote the equivalence class by $\MM(\rho_1)$. For any $\rho\in \ANZ$ note that from \Cref{Moebius space} we have $\MM(\rho)=\MM(Z,\rho)$.
	
	Consider the Type 1 admissible relation (see \Cref{polyhedral structure})
	\begin{equation*}
		R_0^1=\{(1,i),(i,1)\ |\ i\in Z, i\neq1\}.
	\end{equation*}
	In view of \Cref{1st Thm}, for any $\rho\in \ANZ$ there exists $\rho_0\in\MM(\rho)=\MM(Z,\rho)$ such that 
	\begin{equation*}
		R_{\rho_0}=\{(i,j)\in Z\times Z\ |\ \rho(i,j)=1\}=R_0^1.
	\end{equation*} 
	Now, consider the following subset $\mathscr{A}^1(Z)\coloneqq \{ \rho \in \ANZ\ |\ R_{\rho}=R_0^1\ \}\subseteq \ANZ$. Then, every Moebius equivalence class $\MM(\rho)$ intersects the subset $\mathscr A^1(Z)$.
	Next, we define the following:
	\begin{subdefinition}[{\bf Big-teichmuller space}]\label{big-Teich}
		Given a finite set $Z$, consider $\ANZ$, the set of antipodal functions on $Z$. The big-Teichmuller space $\TT(Z)$ is defined as the set of Moebius equivalence  classes in $\ANZ$,
		\begin{equation*}
			\TT(Z)\coloneqq \faktor{\ANZ}{\sim_M}.
		\end{equation*}
	\end{subdefinition}
	\noindent From the discussion above, we have the following map,
	\begin{equation*}
		\begin{split}
			\FF: \mathscr{A}^1(Z)&\to \TT(Z)\\
			\rho\ &\mapsto\ \MM(\rho)
		\end{split} 
	\end{equation*}
	is surjective. As a result we can quotient $\mathscr{A}^1(Z)$ by Moebius equivalence relation to get the bijection,
	\begin{equation}\label{bijection1}
		\begin{split}
			\tilde{\FF}: \left(\faktor{\mathscr{A}^1(Z)}{\sim_M}\right)&\to \TT(Z)\\
			[\rho]\quad\quad  &\mapsto \ \MM(\rho).
		\end{split}
	\end{equation}
	
	For any separating function $\rho$ on $Z$ we have the $m\times m$ matrix representation 
	\begin{equation*}
		M_\rho\coloneqq [\rho(i,j)]_{ij}
	\end{equation*} which is symmetric, having diagonal entries $0$ and off-diagonal entries positive. For $\rho\in\ANZ$, off-diagonal entries of $M_\rho$ are positive less or equal to $1$, with at least one entry in every row being $1$. If $\rho\in \mathscr{A}^1(Z)$ then the symmetric matrix $M_{\rho}$ has the following form
	\begin{equation}\label{matrix form}
		\begin{bmatrix}
			0 & 1 & 1 & \cdots & 1 \\
			1 & 0 & * & \cdots & * \\
			1 & * & 0 & \cdots & * \\
			\vdots & \vdots & \vdots & \ddots & \vdots \\
			1 & * & *& \cdots  & 0
		\end{bmatrix}
	\end{equation}
	where the $*$-marked entries are strictly between 0 and 1. It is clear that, $\rho\in \mathscr{A}^1(Z)$ is determined by the $(1,1)$-minor of the matrix $M_\rho$ denoted by $(M_\rho)_{1,1}$. Let us consider the ordered $(m(m-1)/2)$-tuple
	\begin{equation*}
		S_{1,1}(M_{\rho})=(\ \rho(2,3),\ \rho(2,4),\cdots,\ \rho(m-1,m)\ ),
	\end{equation*} 
	where row-wise entries of $(M_\rho)_{1,1}$ are written in order. From the matrix form \eqref{matrix form} for $\rho\in \mathscr{A}^1(Z)$ it is clear that 
	\begin{equation}\label{auxiliary bijection}
		\begin{tikzcd}[column sep=1,row sep=1]
			{\mathscr{A}^1(Z)} & {} & {(0,1)^{1+p_m}} & {(\ \subseteq\ \  \R^{1+p_m}\ )}\\
			\rho && {S_{1,1}(M_{\rho})} & {}
			\arrow[from=1-1, to=1-3]
			\arrow[maps to, from=2-1, to=2-3]
		\end{tikzcd}
	\end{equation}
	is one-one and onto, where $p_m\coloneqq m(m-3)/2$. 
	
	The following \Cref{Multiple Lemma} provides a characterization for the Moebius equivalence of antipodal functions in $\mathscr{A}^1(Z)$.
	\begin{sublemma}\label{Multiple Lemma}
		Let $\rho_0\in \mathscr{A}^1(Z)$ and $\rho_1$ be a separating function on $Z$, such that $\rho_1(1,i)=1$ for $i\in Z\setminus \{1\}$. Then $\rho_0,\rho_1$ are Moebius equivalent (i.e. $\rho_1\in \mathcal{UM}(Z,\rho_0)$), if and only if there exists $\lambda>0$ such that $\rho_1(i,j)=\lambda\cdot\rho_0(i,j)$ for distinct $i,j\in Z\setminus\{1\}$. Consequently, for any separating function $\rho$ on $Z$ having the same form as $\rho_1$, there exists $\rho'\in \mathscr{A}^1(Z)$ such that $\rho'\in\mathcal{UM}(Z,\rho)$.
		
	\end{sublemma}
	\begin{proof}
		Observe that $M_{\rho_1}$ has a matrix form similar to \eqref{matrix form}, where $*$-marked entries are positive. Suppose $\rho_1\in \mathcal{UM}(Z,\rho_0)$ then by \Cref{GMVT} we have,
		\begin{equation*}
			\begin{bmatrix}
				e^{\tau(1)/2} & 0 & \cdots & 0 \\
				0 & e^{\tau(2)/2} & \cdots & 0 \\
				\vdots & \vdots & \ddots & \vdots \\
				0 & 0 & \cdots  & e^{\tau(m)/2}
			\end{bmatrix}\cdot M_{\rho_1} \cdot
			\begin{bmatrix}
				e^{\tau(1)/2} & 0 & \cdots & 0 \\
				0 & e^{\tau(2)/2} & \cdots & 0 \\
				\vdots & \vdots & \ddots & \vdots \\
				0 & 0 & \cdots  & e^{\tau(m)/2}
			\end{bmatrix} = M_{\rho_2},
		\end{equation*}
		where $\tau=\log\frac{d\rho_2}{d\rho_1}$. Comparing coefficients of first row on either sides we have $\tau(i)=-\tau(1)$, for $i\in Z\setminus\{1\}$. Hence, for $i,j\in Z\setminus\{1\}$, we have 
		\begin{equation*}
			\rho_2(i,j)= e^{\tau(i)/2}\cdot e^{\tau(j)/2}\cdot \rho_1(i,j)=e^{-\tau(1)}\cdot\rho_1(i,j).
		\end{equation*}
		
		\noindent The converse direction involves a straightforward verification of cross-ratios.
		
		For last part choose $\lambda>\max\{\ \rho(i,j)\ |\ 1<i<j\le m\ \}$. 
		Define $\rho'(1,i)=1$ for $i\in Z\setminus\{1\}$ and 
		$$\rho'(i,j)=\frac{1}{\lambda}
		\cdot\rho(i,j), \quad \hbox{for } 1<i<j\le m.$$
	\end{proof}
	
	Define the open simplex of dimension $p_m:=m(m-3)/2$ in $\R^{1+p_m}$
	\begin{equation*}
		\mathring{\Delta}^{p_m} \coloneqq \biggl\{\  x\in \R^{1+p_m}\ |\ x_i>0,\  \sum_{i=1}^{p_m} x_i = 1\ \biggr\}.
	\end{equation*}
	(cf. \cite[Chapter 2.1]{hatcher2002algebraic}). We define the following map,
	\begin{equation*}
		\begin{tikzcd}[column sep=1,row sep=1]
			{\Psi\ :} & {\left(\faktor{\mathscr{A}^1(Z)}{\sim_M}\right)} & {} & {\mathring{\Delta}^{p_m}}\\
			{} & {[\rho]} & {} & {\dfrac{S_{1,1}(M_\rho)}{\|S_{1,1}(M_\rho)\|_1}}.
			\arrow[from=1-2, to=1-4]
			\arrow[maps to, from=2-2, to=2-4]
		\end{tikzcd}
	\end{equation*}
	Here $$\|S_{1,1}(M_\rho)\|_1=\sum_{1<i<j\le m}\rho(i,j)$$ is the sum of upper triangular entries of the matrix $(M_\rho)_{1,1}$.
	\begin{sublemma}\label{bijection2}
		$\Psi$ defined above is well-defined and is a bijection. 
	\end{sublemma}
	\begin{proof}
		By \Cref{Multiple Lemma}, $\Psi$ is well-defined and one-one, that is for $\rho_1,\rho_2\in \mathscr{A}^1(Z)$, $\rho_1\sim_M \rho_2$ if and only if 
		\begin{equation*}
			\dfrac{S_{1,1}(M_{\rho_1})}{\|S_{1,1}(M_{\rho_1})\|_1}=\dfrac{S_{1,1}(M_{\rho_2})}{\|S_{1,1}(M_{\rho_2})\|_1}.
		\end{equation*} In view of the bijection defined in \ref{auxiliary bijection}, for any $x\in\mathring{\Delta}^{p_m}\subseteq (0,1)^{1+p_m}$, there exists $\rho\in \mathscr{A}^1(Z)$ such that $S_{1,1}(M_\rho)=x$. Hence, $\Psi$ is onto. 
	\end{proof}
	
	\begin{subrmk}\label{parametrization}
		As a consequence of \Cref{bijection2}, and the bijection \eqref{bijection1} we have the parametrization $\Phi\colon\TT(Z)\to \mathring{\Delta}^{p_m}$,
		\begin{equation}
			\begin{tikzcd}[column sep=15]
				{\TT(Z)} && {\left(\faktor{\mathscr{A}^1(Z)}{\sim_M}\right)} && {\mathring{\Delta}^{p_m}}
				\arrow["{\tilde{\FF}^{-1}}", from=1-1, to=1-3]
				\arrow["\Phi", curve={height=-40pt}, from=1-1, to=1-5]
				\arrow["\Psi", from=1-3, to=1-5]
			\end{tikzcd}
		\end{equation}
		Moreover, for $\rho\in \ANZ$ there exists unique representative $\tilde{\rho}\in \MM(\rho)\cap \mathscr{A}^1(Z)$ such that $\|S_{1,1}(M_{\tilde{\rho}})\|_1=1$. Also, we have $\Phi(\MM(\rho))= S_{1,1}(M_{\tilde{\rho}})$. 
		From the calculation \eqref{ray} and by \Cref{GMVT} we have explicit description of $\tilde{\rho}$,
		\begin{equation*}
			\tilde{\rho}(i,j)=\frac{\rho(i,j)}{\rho(1,i)\cdot\rho(1,j)}\Bigg(\sum_{1< k<l \le m}\frac{\rho(k,l)}{\rho(1,k)\cdot\rho(1,l)}\Bigg)^{-1}
		\end{equation*}
		for distinct $i,j\in Z\setminus\{1\}$ with $\tilde{\rho}(1,i)=1$ for $i\in Z\setminus\{1\}$, and the values are independent of the choice of the representative of the class $\MM(\rho)$. Thus, we have an explicit description of the parametrization $\Phi$.
	\end{subrmk} 
	
	Next we equip $\TT(Z)$ with a function $\ d_{\hbox{M\"ob}}\colon \TT(Z)\times\TT(Z) \to \R\ $ defined by,
	\begin{equation*}
		d_{\hbox{M\"ob}}(\MM(\rho),\MM(\rho'))\coloneqq \sup_{\xi,\eta,\xi',\eta'}\log \Bigg(\frac{[\xi,\xi',\eta,\eta']_\rho}{[\xi,\xi',\eta,\eta']_{\rho'}}\Bigg)	
	\end{equation*}
	where supremum is taken over distinct $\xi,\eta,\xi',\eta'$ in $Z$,	for $\MM(\rho),\MM(\rho')\in \TT(Z)$.
	\begin{sublemma}
		$d_{\hbox{M\"ob}}$ is a metric on $\TT(Z)$. 
	\end{sublemma}
	\begin{proof}
		For distinct $\xi,\eta,\xi',\eta'\in Z$ and $\rho\in \ANZ$, from \Cref{cross-ratio} we have,
		\begin{equation*}
			[\xi,\xi',\eta,\eta']_\rho\cdot [\xi,\xi',\eta',\eta]_\rho=1.
		\end{equation*} 
		From this it follows that $d_{\hbox{M\"ob}}$ is non-negative and symmetric. 
		Observe that $d_{\hbox{M\"ob}}(\MM(\rho),\MM(\rho'))=0$ if and only if $\rho\sim_M \rho'$.
		Also, for all distinct $\xi,\xi',\eta,\eta'\in Z$ and $\rho_1,\rho_2,\rho_3\in \ANZ$,
		\begin{equation*}
			\begin{split}
				\log \Bigg(\frac{[\xi,\xi',\eta,\eta']_{\rho_3}}{[\xi,\xi',\eta,\eta']_{\rho_1}}\Bigg)&=\log \Bigg(\frac{[\xi,\xi',\eta,\eta']_{\rho_3}}{[\xi,\xi',\eta,\eta']_{\rho_2}}\Bigg)+\log \Bigg(\frac{[\xi,\xi',\eta,\eta']_{\rho_2}}{[\xi,\xi',\eta,\eta']_{\rho_1}}\Bigg)\\
				&\le d_{\hbox{M\"ob}}(\MM(\rho_3),\MM(\rho_2))+d_{\hbox{M\"ob}}(\MM(\rho_2),\MM(\rho_1)).\\
			\end{split}		
		\end{equation*}
		Thus, $d_{\hbox{M\"ob}}$ satisfies the triangle inequality and, therefore, is a metric on $\TT(Z)$.
	\end{proof}
	
	\begin{sublemma}\label{geodesic}
		Given $Z$ finite set. $(\TT(Z),d_{\hbox{M\"ob}})$ is a geodesic, geodesically complete metric space.
	\end{sublemma}
	\begin{proof}
		Let $\MM(\rho_0),\MM(\rho_1)\in \TT(Z)$, where $\rho_0$, and $\rho_1$ are representatives as specified in \Cref{parametrization}. Let $d=d_{\hbox{M\"ob}}(\MM(\rho_1),\MM(\rho_0))$.  Define the path $\gamma:[0,d]\to \TT(Z)$ by convex combination,
		\begin{equation*}
			\gamma(t)\coloneqq [(\rho_1)^{t/d}\cdot(\rho_0)^{(1-t/d)}]_M.
		\end{equation*}
		Note that for $t\in[0,1]$ we have $\rho_t\coloneqq (\rho_1)^{t/d}\cdot(\rho_0)^{(1-t/d)} \in \mathscr{A}^1(Z)$.
		Observe that for distinct $\xi,\xi',\eta,\eta'\in Z$
		\begin{equation*}
			\frac{[\xi,\xi',\eta,\eta']_{\rho_t}}{[\xi,\xi',\eta,\eta']_{\rho_0}}=\Bigg(\frac{[\xi,\xi',\eta,\eta']_{\rho_1}}{[\xi,\xi',\eta,\eta']_{\rho_0}}\Bigg)^{t/d},
		\end{equation*}
		thus $d_{\hbox{M\"ob}}(\gamma(0),\gamma(t))=t$. Thus, $\gamma$ is a geodesic joining $\MM(\rho_0)$ to $\MM(\rho_1)$. 
		
		Furthermore, for any separating function $\rho_0'$ and $\rho_1'$ Moebius equivalent to $\rho_0$ and $\rho_1$ respectively then for all $t\in\R$ the separating function
		\begin{equation*}
			\rho_t'\coloneqq (\rho_1')^{t/d}\cdot(\rho_0')^{(1-t/d)}
		\end{equation*}is Moebius equivalent to the separating $\rho_t$ as defined above. Now $\rho_t(1,i)=1$ for all $i\in Z\setminus\{1\}$ and $t\in \R$. However, $\rho_t$ need not be an antipodal for $t\in \R\setminus[0,1]$. By \Cref{Multiple Lemma} there exists $\tilde{\rho_t}\in \mathscr{A}^1(Z)$ such that $\rho_t$ is Moebius equivalent to $\tilde{\rho_t}$. So we can extend $\gamma:\R\to \TT(Z)$ with $\gamma(t)=[\tilde{\rho_t}]_M$ for $t\in \R\setminus[0,1]$. Let $t>s$ in $\R$, for distinct $\xi,\xi',\eta,\eta'\in Z$  we have,
		\begin{equation*}
			\frac{[\xi,\xi',\eta,\eta']_{\rho_t}}{[\xi,\xi',\eta,\eta']_{\rho_s}}=\frac{([\xi,\xi',\eta,\eta']_{\rho_1})^{t/d}\cdot([\xi,\xi',\eta,\eta']_{\rho_0})^{(1-t/d)}}{([\xi,\xi',\eta,\eta']_{\rho_1})^{s/d}\cdot([\xi,\xi',\eta,\eta']_{\rho_0})^{(1-s/d)}}=\Bigg(	\frac{[\xi,\xi',\eta,\eta']_{\rho_1}}{[\xi,\xi',\eta,\eta']_{\rho_0}}\Bigg)^{(t-s)/d}
		\end{equation*}
		which implies $d_{\hbox{M\"ob}}(\gamma(t),\gamma(s))=(t-s)$. Thus, $(\TT(Z),d_{\hbox{M\"ob}})$ is geodesically complete.
	\end{proof}
	
	\begin{subrmk}
		We have defined the big-Teichmuller space for finite set $Z$. One can naturally extend the definition of big-Teichmuller spaces for any compact metrizable space $Z$ that admits antipodal functions. Since the argument does not depend on the finiteness of $Z$, it follows that $(\TT(Z), d_{\text{M\"ob}})$ is a metric space. However, whether this space is geodesic in general remains unclear. The main obstacle to this result is that it is not definitively known whether every separating function on any $Z$ has a Moebius equivalent antipodal function; for $Z$ finite, this is indeed true.
	\end{subrmk}
	Now we prove the following:
	\begin{subtheorem}\label{big-Teich theorem}
		Given a finite set $Z$ of cardinality $m$, the big-Teichmuller space $(\TT(Z),d_{\hbox{M\"ob}})$ is a geodesic metric space. Moreover, it is geodesically complete and homeomorphic to the open simplex of dimension $$p_m\coloneqq \frac{m(m-3)}{2}$$ equipped with topology of $\R^{1+p_m}$.     
	\end{subtheorem}
	\begin{proof}
		We know from \Cref{parametrization} that, there is a parametrization $$\Phi:\TT(Z) \to \mathring{\Delta}^{p_m}.$$ We have already seen that $(\TT(Z),d_{\hbox{M\"ob}})$ is a geodesic, geodesically complete metric space in \Cref{geodesic}. Now, we will demonstrate that $\Phi$ is indeed, a homeomorphism between $(\TT(Z), d_{\text{M\"ob}})$ and the open simplex $\mathring{\Delta}^{p_m} \subseteq \R^{p_m}$. 
		
		Observe that, given $\MM(\rho_n)\in \TT(Z)$ a sequence, $$\MM(\rho_n)\xrightarrow{n\to\infty} \MM(\rho_0)\in\TT(Z)$$ if and only if for distinct $\xi,\xi',\eta,\eta'\in Z$ 	the cross-ratios $[\xi,\xi',\eta,\eta']_{\rho_n}\to [\xi,\xi',\eta,\eta']_{\rho_0}$.
		
		Let $x_n\in\mathring{\Delta}^{p_m}$ such that $x_n\to x_0\in \mathring{\Delta}^{p_m}$. There exists unique $\rho_0\in \mathscr{A}^1(Z)$ such that $$\Phi(\MM(\rho_0))=S_{1,1}(M_{\rho_0})=x_0.$$ Similarly, for each $n$, there exists unique $\rho_n\in \mathscr{A}^1(Z)$ such that $$\Phi(\MM(\rho_n))=S_{1,1}(M_{\rho_n})=x_n$$ (see \Cref{parametrization}). Now $x_n\to x_0$ implies $M_{\rho_n}\to M_{\rho_0}$ and thus the cross-ratios
		$$[\xi,\xi',\eta,\eta']_{\rho_n}\to [\xi,\xi',\eta,\eta']_{\rho_0}$$ for distinct $\xi,\xi',\eta,\eta'\in Z$. Thus, $\Phi^{-1}$ is continuous. 
		
		Now suppose $\MM(\rho_n)\in \TT(Z)$ a sequence, such that $\MM(\rho_n)\to \MM(\rho_0)\in\TT(Z)$. Representatives $\rho_0$ and $\rho_n$'s are chosen so that \begin{equation*}
			\Phi(\MM(\rho_0))=S_{1,1}(M_{\rho_0})\ \hbox{ and }\ \Phi(\MM(\rho_n))=S_{1,1}(M_{\rho_n})
		\end{equation*} respectively, as specified in \Cref{parametrization}.
		We want to show $S_{1,1}(\rho_n)\to S_{1,1}(\rho_0)$ which is same as showing $\rho_n(i,j)\to \rho_0(i,j)$, for $i,j\in Z$. For each $i,j\in Z$ we have $\rho_n(i,j)\in[0,1]$, hence a bounded sequence. Passing to a subsequence we let $\rho_n(i,j)\to \alpha(i,j)\in [0,1]$. Note that $\alpha(i,i)=0$ and $\alpha(1,j)=1$ for $j\in Z\setminus\{1\}$. Moreover, 
		\begin{equation*}
			1=\|S_{1,1}(M_{\rho_n})\|_1\to \sum_{1< i<j \le m} \alpha(i,j)=1
		\end{equation*}
		Hence there exists distinct $l,k\in Z\setminus \{1\}$, so that $\alpha(l,k)>0$. For distinct $i,j\in Z\setminus\{1\}$ such that $(i,j)\not\in\{(k,l),(l,k)\}$ we have
		\begin{equation*}
			\begin{split}
				\frac{\alpha(i,j)}{\alpha(k,l)}=\lim_{n\to\infty} \frac{\rho_n(i,j)}{\rho_n(k,l)}=&\lim_{n\to\infty}\  [1,i,k,j]_{\rho_n}\cdot [1,i,l,k]_{\rho_n}\\
				=&[1,i,k,j]_{\rho_0}\cdot [1,i,l,k]_{\rho_0}>0
			\end{split}
		\end{equation*}
		Thus, $\alpha(i,j) > 0$ for distinct $i, j \in Z \setminus \{1\}$. Consequently, $\alpha \in \mathscr{A}^1(Z)$, and $\alpha \sim_M \rho_0$. Note that we also have $\|S_{1,1}(M_\alpha)\|_1 = 1$. By the uniqueness of the representative as per \Cref{parametrization}, we conclude that $\alpha = \rho_0$. Finally, observe that for any subsequence of $\rho_n$, we can extract a further subsequence such that $\rho_n$ converges to $\rho_0$. Hence, $\rho_n$ converges to $\rho_0$, which implies that $\Phi$ is continuous.
	\end{proof}
	\medskip
	
	\subsection{Teichmuller space and Mapping class group}
	Motivated by the classical definition of the Teichmuller space for Riemann surfaces, this section defines the Teichmuller space for the Moebius space $\MM(\rho) \in \TT(Z)$. Before proceeding, observe that since $Z$ is finite, any homeomorphism
	\begin{equation*}
		H \colon \MM(\rho_1) = \MM(Z,\rho_1) \to \MM(Z,\rho_2) = \MM(\rho_2)
	\end{equation*}
	naturally extends to a homeomorphism $\tilde{H}\colon \overline{\MM(\rho_1)}\to \overline{\MM(\rho_2)}$ between the Gromov compactifications, where the Gromov boundary is naturally identified with $Z$ (by \Cref{gromov ip}). Also, the extension $h\coloneqq\tilde{H}|_Z \colon (Z,\rho_1) \to (Z,\rho_2)$ is a bijection between the Gromov boundaries. Throughout this section, we will consider homeomorphisms between Moebius spaces of $Z$, denoted by uppercase letters, while maps between the Gromov boundaries $Z$ will be denoted by lowercase letters.
	
	Consider the set of markings of $\MM(\rho)$
	\begin{equation*}
		\mathcal S\big(\MM(\rho)\big)\coloneqq \bigg\{\big(\MM(\rho),F\big)\ |\ F\colon \MM(\rho_0)\xrightarrow{\mathrm{Homeo}} \MM(\rho)\ \bigg\}
	\end{equation*}
	define an equivalence relation $``\sim"$ on $\mathcal S(\MM(\rho))$, where $(\MM(\rho_1),F_1) \sim (\MM(\rho_2),F_2)$
	if and only if 
	\begin{equation}
		\begin{tikzcd}[sep=normal]
			& {\MM(\rho_1)} &&& Z \\
			{\MM(\rho_0)} &&\hbox{with boundary maps}& Z \\
			& {\MM(\rho_2)} &&& Z
			\arrow["\exists\ G", dashed, from=1-2, to=3-2]
			\arrow["g", from=1-5, to=3-5]
			\arrow["{F_1}", from=2-1, to=1-2]
			\arrow["{F_2}"', from=2-1, to=3-2]
			\arrow["{f_1}", from=2-4, to=1-5]
			\arrow["{f_2}"', from=2-4, to=3-5]
		\end{tikzcd}
	\end{equation}
	there exists an isometry $ G \colon \MM(Z,\rho_1) \to \MM(Z,\rho_2) $ such that the boundary maps satisfy $	g=\left(f_2\circ f_1^{-1}\right).$
	
	\medskip
	
	\noindent Now we define the Teichmuller space for a Moebius space $\MM(Z,\rho)$.
	
	\begin{definition}[{\bf Teichmuller space}]\label{Teichmuller space}
		Given $\MM(Z,\rho_0)=\MM(\rho_0)\in \TT(Z)$ we define the Teichmuller space,
		\begin{equation}
			\TT(\MM(\rho_0))\coloneqq \faktor{\mathcal S\big(\MM(\rho_0)\big)}{\sim}
		\end{equation}
		set of equivalence classes in $\mathcal{S}\big(\MM(\rho_0)\big)$ given by the equivalence relation ``$\sim$" defined above.
	\end{definition}
	Here, we note a few observations. 
	\begin{enumerate}[(i)]
		\item Suppose $(\MM(\rho), F_1)$ and $(\MM(\rho), F_2)$ are such that the homeomorphisms $F_1\colon \MM(\rho_0) \to \MM(\rho)$ and $F_2\colon \MM(\rho_0) \to \MM(\rho)$ have $f_1 = f_2$, meaning the boundary maps are the same. Then $(\MM(\rho), F_1) \sim (\MM(\rho), F_2)$. This is because we can take $G\colon \MM(\rho) \to \MM(\rho)$ to be the identity map.
		\\  
		
		\item For every $[(\MM(\rho),F)]\in \TT([\rho_0])$ there exists $(\MM(\tilde{\rho}),\tilde{F})\in [(\MM(\rho),F)]$ such that the extension of $\tilde{F}$ to the Gromov boundary, $\tilde{f}$, is the identity map of $Z$. Let $f\colon Z\to Z$ be the extension of $F$. Take $\tilde{\rho}(\cdot,\cdot)=\rho(f(\cdot),f(\cdot))$ and set $$\tilde{F}=((f^{-1})_*\circ F)\colon \MM(\rho_0)\to \MM(\tilde{\rho}).$$
		Here $f^{-1} \colon (Z,\rho) \to (Z,\tilde{\rho})$ is a Moebius homeomorphism, and therefore the push-forward $G=(f^{-1})_*:\MM(Z,\rho)\to \MM(Z,\tilde{\rho})$ is an isometry, with boundary map $(f^{-1})$; see \cite[Proposition 6.4]{biswas2024quasi}. Thus, we have $(\MM(\tilde{\rho}),\tilde{F})\sim(\MM(\rho),F)$.\\
		
		\item Suppose $(\MM(\rho_1), F_1) \sim (\MM(\rho_1), F_2)$ are such that the homeomorphisms $F_1\colon \MM(\rho_0) \to \MM(\rho)$ and $F_2\colon \MM(\rho_0) \to \MM(\rho)$ have $f_1 = f_2$, meaning the same boundary maps. Then there exists an isometry $G:\MM(\rho_1)\to\MM(\rho_2)$ with the boundary map $$g = \big(f_1 \circ f_1^{-1}\big) = id|_Z,$$ and by \cite[Theorem 1.5]{biswas2021quasi} we have $G=g_*$. Therefore, $\MM(\rho_1)=\MM(\rho_2)$ with $G$ identity map.\\

		\item Suppose $\MM(\rho_1),\ \MM(\rho_2)\in \TT(Z)$ be such that they are homeomorphic via a homeomorphism $H_0\colon\MM(\rho_1)\to \MM(\rho_2)$. Consider the map,
		\begin{equation}\label{identification of Teichmuller space}
			\begin{split}
				\TT(\MM(\rho_1))&\xrightarrow{H_*} \quad\TT(\MM(\rho_2))\\
				[(\MM(\rho),F)]&\mapsto\ [(\MM(\rho),F\circ H_0^{-1})].
			\end{split}
		\end{equation} It is straightforward to see that the map $H_*$ is a well-defined bijection. Thus, the two Teichmuller spaces $\TT(\MM(\rho_1))$ and $\TT(\MM(\rho_2))$ are identified via $H_*$.\\
	\end{enumerate} 
	
	\noindent Next, we encounter the following equivalence relation. We say two Moebius spaces $\MM(\rho_0),\MM(\rho_1)\in \TT(Z)$ are {\it homeomorphism equivalent} if there exists a homeomorphism $F\colon \MM(Z, \rho_0) \to \MM(Z, \rho_1)$ with the boundary map $f\colon Z\to Z$ is identity, denote by $\MM(\rho_0)\sim_H\MM(\rho_1)$. The equivalence class is denoted by $[\MM(\rho_1)]_H$,
	\begin{equation*}
		[\MM(\rho_0)]_H\coloneqq \big\{\ \MM(\rho)\in \TT(Z)\ |\ \hbox{ there exists } F\colon \MM(\rho_0)\xrightarrow{\mathrm{Homeo}} \MM(\rho),\hbox{ with } (f\colon Z\to Z\big)=id|_Z\}.
	\end{equation*}
	
	There is a one-one correspondence between $[\MM(\rho_0)]_H$ and $\TT(\MM(\rho_0))$, given by the map,
	\begin{equation*}
		\begin{split}
			[\MM(\rho_0)]_H &\to\ \TT\big(\MM(\rho_0)\big)\\
			\MM(\rho) \quad &\mapsto [(\MM(\rho),F)]
		\end{split}
	\end{equation*}
	where $F \colon \MM(\rho_0) \to \MM(\rho)$ is a homeomorphism with boundary map $f\colon Z \to Z$ is identity. From the first observation, we see that the map is well-defined. It is also onto and one-to-one, as evident from the second and third observations.
	
	\begin{definition}[Mapping class group]
		Given $\MM(\rho)\in\TT(Z)$, consider the homeomorphism group
		\begin{equation*}
			\mathrm{Homeo}(\MM(\rho))\coloneqq\{\ F\colon \MM(\rho)\xrightarrow{\mathrm{Homeo}}\MM(\rho)\ \}
		\end{equation*} and the normal subgroup
		\begin{equation*}
			\mathrm{Homeo}_0(\MM(\rho))\coloneqq\{\ F\in \mathrm{Homeo}(\MM(\rho))\ | \hbox{ the boundary map } (f\colon Z\to Z) \hbox{ is identity }\}.
		\end{equation*}
		The mapping class group for $\MM(\rho)$ is defined as the quotient group
		\begin{equation*}
			\mathrm{MCG}\big(\MM(\rho)\big)\coloneqq \dfrac{\mathrm{Homeo}(\MM(\rho))}{\mathrm{Homeo}_0(\MM(\rho))}.
		\end{equation*}
	\end{definition}
	\noindent Observe that there exists a natural homomorphism from $\mathrm{Homeo}(\MM(\rho))$ to $\mathrm{Perm}(Z)$, the group of bijections (or permutations) of $Z$,
	\begin{equation*}
		\begin{split}
			\mathrm{Homeo}\big(\MM(\rho)\big)&\xrightarrow{\Lambda} \mathrm{Perm}(Z)\cong S_m\\
			F \quad &\mapsto \quad f
		\end{split}
	\end{equation*}
	where again $f$ is the boundary extension of the homeomorphism $F$. We have $Ker(\Lambda)= \mathrm{Homeo}_0(\MM(\rho))$ thus $\mathrm{MCG}(\MM(\rho))$ is isomorphic to a subgroup of the $\mathrm{Perm}(Z)\cong S_m$.
	\medskip
	
	For $\rho_0\in \ANZ$, we define a group action of $\mathrm{MCG}(\MM(\rho_0))$ on $\TT(\MM(\rho_0))$,
	\begin{equation*}
		\begin{split}
			\mathrm{MCG}(\MM(\rho_0))\times \TT(\MM(\rho_0)) &\to \TT(\MM(\rho_0))\\
			(\ [H]\ ,\ [(\MM(\rho),F)]\ ) &\mapsto [(\MM(\rho),F\circ H)]
		\end{split}
	\end{equation*}
	It is straightforward to see that the action is well-defined independent of the choice of representatives.
	
	\bigskip
	
	\section{Maximal Gromov hyperbolic space with $4$ point boundary}\label{4 points}
	
	This section will discuss the different polyhedral structures of maximal Gromov Hyperbolic spaces when they have $4$ points on the Gromov boundary. Specifically, we will study the Moebius spaces $\MM(Z,\rho)$ associated when antipodal space $(Z,\rho)$ contains $4$ elements. Furthermore, we will comment on the big-Teichmuller space and Teichmuller space in this case.

	Let $Z=\{1,2,3,4\}$. Then the dimension of the big-Teichmuller space $\TT(Z)$ is $p_4=2$ (by \Cref{big-Teich theorem}) and is parametrized by open $2$-simplex $$\mathring{\Delta}^2=\{(\mu,\lambda,\nu)\in \R^3\ |\ \mu,\lambda,\nu>0,\ \mu+\lambda+\nu=1\ \}$$
	via the parametrization $\Phi\colon \TT(Z)\to \mathring{\Delta}^2$. 
	
	For $\MM(\rho)\in \TT(Z)$ where $\rho$ is the unique representative as in \Cref{parametrization} then the corresponding matrix is of the form
	\begin{equation*}
		M_\rho=\begin{bmatrix}
			0 & 1 & 1 & 1 \\
			1& 0 & \mu & \lambda \\
			1 & \mu & 0 & \nu \\
			1 & \lambda & \nu  & 0
		\end{bmatrix}
	\end{equation*} where $\mu=\rho(2,3),\ \lambda=\rho(2,4),\ \nu=\rho(3,4)$, $\mu+\lambda+\nu=1$ and $\Phi(\MM(\rho))=(\mu,\lambda,\nu)$. Recall $a_{ij}= \log (\rho(i,j)^2)$ (as in \Cref{polyhedral structure}). So in this case, $a_{23}=2\log \mu$, $a_{24}=2\log\lambda$, $a_{34}=2\log \nu$ and $a_{1i}=0$ for all $i\in Z\setminus\{1\}$.
	The highest possible dimensional of a polyhedral cell in any $\MM(Z,\rho)\in \TT(Z)$ is at most $2$. First, we will try to understand those.
	We consider three situations:  
	\begin{enumerate}
		\item Consider the relation $R=\{(1,2),(2,1),(3,4),(4,3)\}$ with graph $\Gamma_R$ 
		\begin{figure}[H]
			\centering
			\begin{tikzpicture}[scale=0.5,line cap=round,line join=round,>=triangle 45,x=1cm,y=1cm]
				\clip(1.75,1.75) rectangle (4.25,4.25);
				\draw [line width=0.4pt] (2,4)-- (4,4);
				\draw [line width=0.4pt] (2,2)-- (4,2);
				\begin{scriptsize}
					\draw [fill=black] (2,4) circle (2.5pt);
					\draw[color=black] (1.6,4.1) node {$1$};
					\draw [fill=black] (4,4) circle (2.5pt);
					\draw[color=black] (4.4,4.1) node {$2$};
					\draw [fill=black] (2,2) circle (2.5pt);
					\draw[color=black] (1.6,1.9) node {$4$};
					\draw [fill=black] (4,2) circle (2.5pt);
					\draw[color=black] (4.4,1.9) node {$3$};
				\end{scriptsize}
			\end{tikzpicture}
		\end{figure}\noindent which has two even cycles. So the cell $C(R)^*$, if non-empty, must be $2$-dimensional. Then the closed cell $C(R)$ in $\MM(Z,\rho)$,
		\begin{equation*}
			C(R)=\bigg\{\ \big(t_1,-t_1,t_2,-t_3-a_{34}\big)\ \big|\ a_{24}-a_{34}\le t_1+t_2\le 0,\ a_{23}\le t_1-t_2\le a_{34}\ \bigg\}
		\end{equation*} is non-empty if and only if $a_{23}\le a_{34}$ and $a_{24}\le a_{34}$, if and only if $\max\{\mu,\lambda\}\le \nu$.\\
		
		\item Similarly, for $R=\{ (1,3),(3,1),(2,4),(4,2) \}$ with the graph $\Gamma_R$ being,
		\begin{figure}[H]
			\centering
			\begin{tikzpicture}[scale=0.5,line cap=round,line join=round,>=triangle 45,x=1cm,y=1cm]
				\clip(1.75,1.75) rectangle (4.25,4.25);
				\draw [line width=0.4pt] (2,4)-- (4,2);
				\draw [line width=0.4pt] (2,2)-- (4,4);
				\begin{scriptsize}
					\draw [fill=black] (2,4) circle (2.5pt);
					\draw[color=black] (1.6,4.1) node {$1$};
					\draw [fill=black] (4,4) circle (2.5pt);
					\draw[color=black] (4.4,4.1) node {$2$};
					\draw [fill=black] (2,2) circle (2.5pt);
					\draw[color=black] (1.6,1.9) node {$4$};
					\draw [fill=black] (4,2) circle (2.5pt);
					\draw[color=black] (4.4,1.9) node {$3$};
				\end{scriptsize}
			\end{tikzpicture}
		\end{figure}\noindent then the polyhedral cell 
		\begin{equation*}
			C(R)=\bigg\{\ \big(t_1,t_2,-t_1,-t_2-a_{24}\big)\ \big|\ a_{34}-a_{24}\le t_1+t_2\le 0,\ a_{23}\le t_1-t_2\le a_{24}\ \bigg\}
		\end{equation*}
		is non-empty if and only if $a_{23}\le a_{24}$ and $a_{34}\le a_{24}$, if and only if $\max\{\mu,\nu\}\le \lambda$.\\
		
		\item Finally, for the relation $R=\{(1,4),(4,1),(2,3),(3,2)\}$ with $\Gamma_R$ being,
		\begin{figure}[H]
			\centering
			\begin{tikzpicture}[scale=0.5,line cap=round,line join=round,>=triangle 45,x=1cm,y=1cm]
				\clip(1.75,1.75) rectangle (4.25,4.25);
				\draw [line width=0.4pt] (2,4)-- (2,2);
				\draw [line width=0.4pt] (4,4)-- (4,2);
				\begin{scriptsize}
					\draw [fill=black] (2,4) circle (2.5pt);
					\draw[color=black] (1.6,4.1) node {$1$};
					\draw [fill=black] (4,4) circle (2.5pt);
					\draw[color=black] (4.4,4.1) node {$2$};
					\draw [fill=black] (2,2) circle (2.5pt);
					\draw[color=black] (1.6,1.9) node {$4$};
					\draw [fill=black] (4,2) circle (2.5pt);
					\draw[color=black] (4.4,1.9) node {$3$};
				\end{scriptsize}
			\end{tikzpicture}
		\end{figure}
		\noindent then the polyhedral cell
		\begin{equation*}
			C(R)=\bigg\{\ \big(t_1,t_2,-t_2-a_{23},-t_1\big)\ \big|\ a_{34}-a_{23}\le t_1+t_2\le 0,\ a_{24}\le t_1-t_2\le a_{23}\ \bigg\}
		\end{equation*}
		is non-empty if and only if $a_{24}\le a_{23}$ and $a_{34}\le a_{23}$, if and only if $\max\{\lambda,\nu\}\le \mu$.\\
	\end{enumerate}
	
	Observe that one of the three situations must occur. Let us suppose we are in the situation $\max\{\mu,\lambda\}\le\nu$; the other situations can be dealt with similarly. Then we have a non-empty rectangular cell (possibly a degenerate rectangle) $C(R)$ in $\MM(Z,\rho)$ for the Type $2$ relation $R=\{(1,2),(2,1),(3,4),(4,3)\}$. Now we look at the four Type $1$ relations $R_0^i$, for $i\in Z$ with the graphs being 
	\begin{figure}[H]
		\centering
		\begin{tikzpicture}[scale=0.5,line cap=round,line join=round,>=triangle 45,x=1cm,y=1cm]
			\clip(1.75,1.75) rectangle (19.25,4.25);
			\draw [line width=0.4pt] (2,4)-- (2,2);
			\draw [line width=0.4pt] (2,4)-- (4,2);
			\draw [line width=0.4pt] (2,4)-- (4,4);
			\draw [line width=0.4pt] (7,4)-- (9,4);
			\draw [line width=0.4pt] (9,4)-- (7,2);
			\draw [line width=0.4pt] (9,4)-- (9,2);
			\draw [line width=0.4pt] (12,4)-- (14,2);
			\draw [line width=0.4pt] (14,4)-- (14,2);
			\draw [line width=0.4pt] (12,2)-- (14,2);
			\draw [line width=0.4pt] (17,4)-- (17,2);
			\draw [line width=0.4pt] (17,2)-- (19,4);
			\draw [line width=0.4pt] (17,2)-- (19,2);
			\begin{scriptsize}
				\draw [fill=black] (2,4) circle (2.5pt);
				\draw[color=black] (1.6,4.1) node {$1$};
				\draw [fill=black] (4,4) circle (2.5pt);
				\draw[color=black] (4.4,4.1) node {$2$};
				\draw [fill=black] (2,2) circle (2.5pt);
				\draw[color=black] (1.6,1.9) node {$4$};
				\draw [fill=black] (4,2) circle (2.5pt);
				\draw[color=black] (4.4,1.9) node {$3$};
				\draw [fill=black] (7,4) circle (2.5pt);
				\draw[color=black] (6.6,4.1) node {$1$};
				\draw [fill=black] (7,2) circle (2.5pt);
				\draw[color=black] (6.6,1.9) node {$4$};
				\draw [fill=black] (9,2) circle (2.5pt);
				\draw[color=black] (9.4,1.9) node {$3$};
				\draw [fill=black] (9,4) circle (2.5pt);
				\draw[color=black] (9.4,4.1) node {$2$};
				\draw [fill=black] (12,4) circle (2.5pt);
				\draw[color=black] (11.6,4.1) node {$1$};
				\draw [fill=black] (12,2) circle (2.5pt);
				\draw[color=black] (11.6,1.9) node {$4$};
				\draw [fill=black] (14,4) circle (2.5pt);
				\draw[color=black] (14.4,4.1) node {$2$};
				\draw [fill=black] (14,2) circle (2.5pt);
				\draw[color=black] (14.4,1.9) node {$3$};
				\draw [fill=black] (17,4) circle (2.5pt);
				\draw[color=black] (16.6,4.1) node {$1$};
				\draw [fill=black] (17,2) circle (2.5pt);
				\draw[color=black] (16.6,1.9) node {$4$};
				\draw [fill=black] (19,4) circle (2.5pt);
				\draw[color=black] (19.4,4.1) node {$2$};
				\draw [fill=black] (19,2) circle (2.5pt);
				\draw[color=black] (19.4,1.9) node {$3$};
			\end{scriptsize}
		\end{tikzpicture}
	\end{figure} \noindent respectively. From \Cref{polyhedral structure} \eqref{ray}, we can describe the rays associated with these relations as follows:
	\begin{equation*}
		\begin{split}
			C(R_0^1) &= \bigg\{\ \big(t, -t, -t, -t\big)\ \big|\ t \geq \frac{a_{34}}{2} = \log \nu\ \bigg\},\\
			C(R_0^2) &= \bigg\{\ \big(-t,\ t, -(t + a_{23}), -(t + a_{24})\big)\ \big|\ t \geq \frac{a_{34} - a_{23} - a_{24}}{2} = \log \nu - \log \mu - \log \lambda\ \bigg\},\\
			C(R_0^3) &= \bigg\{\ \big(-t, -(t + a_{23}),\ t, -(t + a_{34})\big)\ \big|\ t \geq \frac{-a_{24}}{2} = -\log \lambda\ \bigg\},\\
			C(R_0^4) &= \bigg\{\ \big(-t, -(t + a_{24}), -(t + a_{23}),\ t\big)\ \big|\ t \geq \frac{-a_{23}}{2} = -\log \mu\ \bigg\}.
		\end{split}
	\end{equation*}
	Furthermore, the four rays originate from the four vertices of the rectangle $C(R)$ mentioned above, which are determined by straightforward calculations. The Moebius space $\MM(Z,\rho)$ is the union of the polyhedron $C(R)$ and the four rays.
	
	Now we will see how the structure of $\MM(Z,\rho)\in \TT(Z)$ varies as $(\mu,\nu,\lambda)$ varies in the parameter space $\mathring{\Delta}^2$. We will consider the following six cases.
	\begin{figure}[H]
		\centering
		\includegraphics[width=0.5\linewidth]{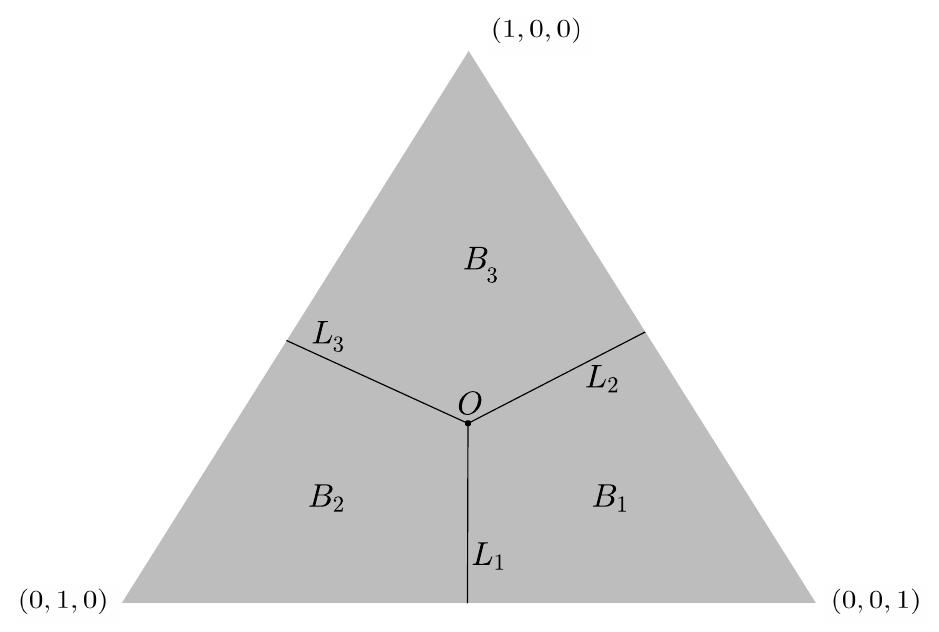}
		\caption{Subdivisions of $\mathring{\Delta}^2$}
		\label{fig:simplex}
	\end{figure}
	
	\begin{enumerate}[(i)]
		\item Suppose $(\mu,\lambda,\nu)\in B_1$ the open region in $\mathring{\Delta}^2$ (see \Cref{fig:simplex}) such that we have strict inequality, $\max\{\mu,\lambda\}<\nu$, then $C(R)^*$ is non-empty with the polyhedral cell $C(R)$ is a $2$-dimensional rectangle with length and breadth being $$(a_{34}-a_{24})= 2\log\Big(\frac{\nu}{\lambda}\Big)\ \hbox{ and }\ (a_{34}-a_{23})=2\log\Big(\frac{\nu}{\mu}\Big).$$ Then $\MM(Z,\rho)$ has the following structure.
		\begin{figure}[H]
			\centering
			\includegraphics[width=0.5\linewidth]{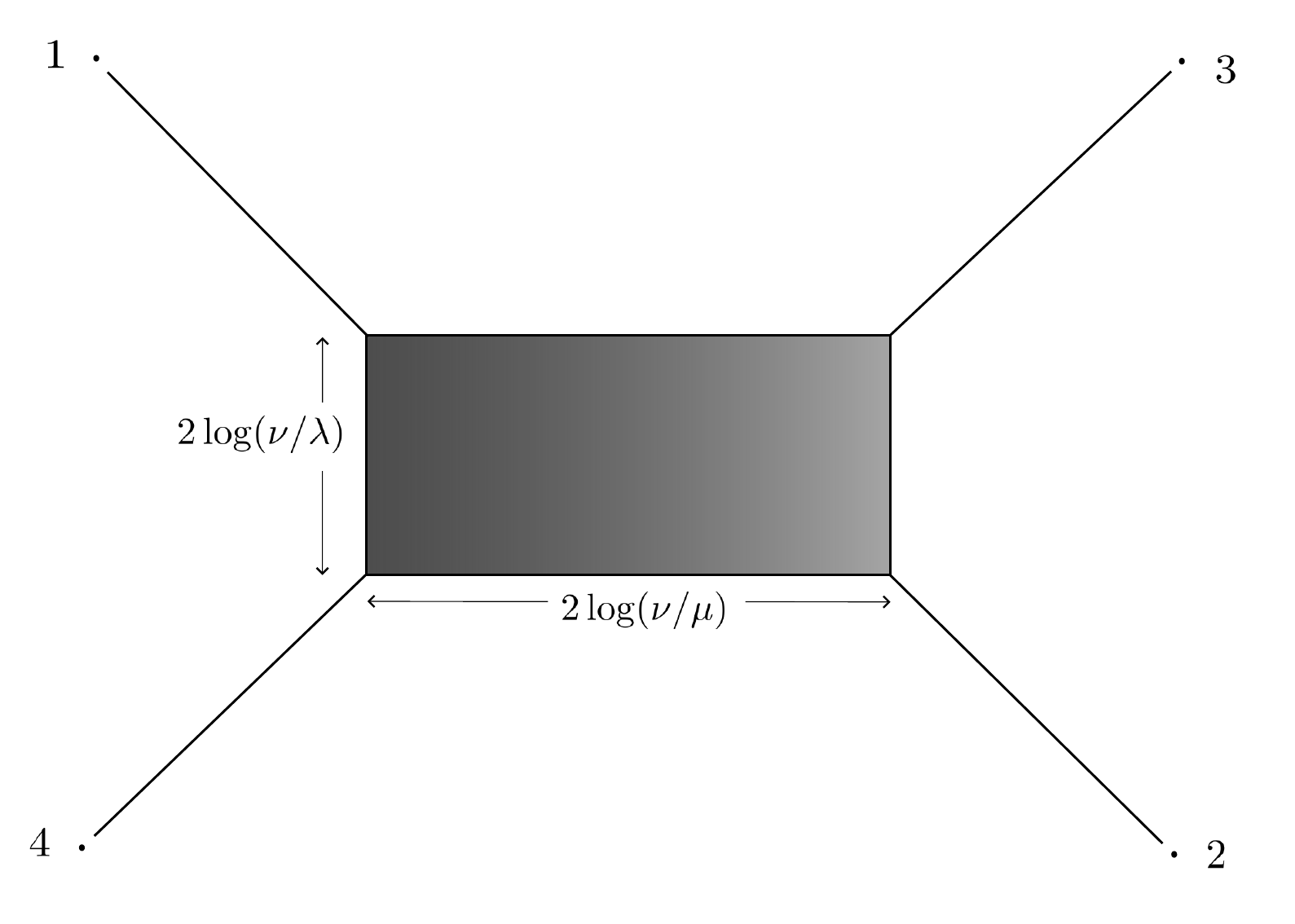}
			\caption{}
			\label{fig:4 boundary points 1}
		\end{figure}
		\item Similarly, if we have $(\mu,\lambda,\nu)\in B_3$, that is $\max\{\mu,\nu\}<\lambda$. So we are in the situation (2) with $C(R)^*$ non-empty. Then $\MM(Z,\rho)$ have the following structure
		\begin{figure}[H]
			\centering
			\includegraphics[width=0.5\linewidth]{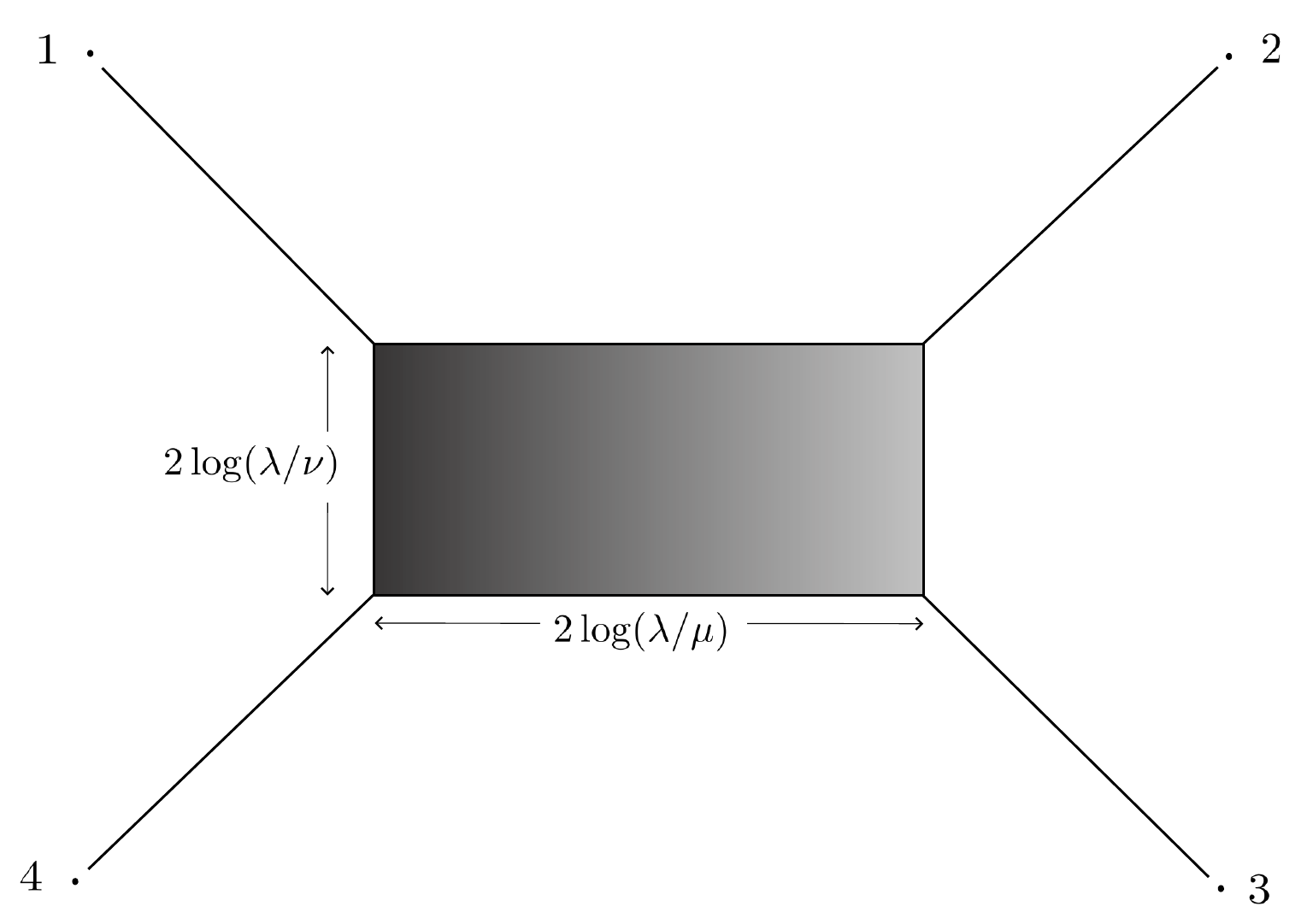}
			\caption{}
			\label{fig:4 boundary points 2}
		\end{figure}
		\item Also, for $(\mu,\lambda,\nu)\in B_2$, that is $\max\{\lambda,\nu\}< \mu$, situation (3) with $C(R)^*$ non-empty. The structure of $\MM(Z,\rho)$ is the following:
		\begin{figure}[H]
			\centering
			\includegraphics[width=0.5\linewidth]{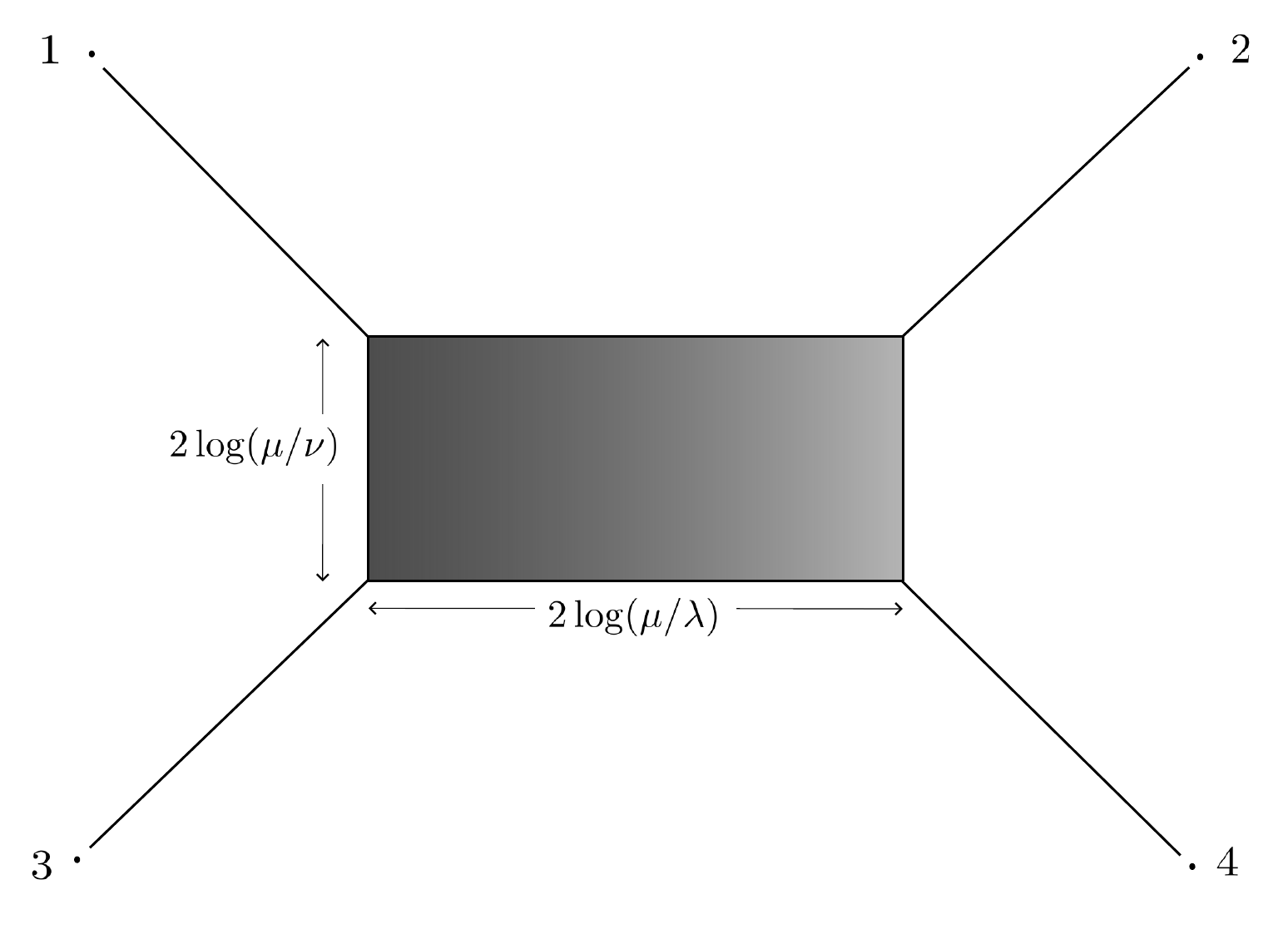}
			\caption{}
			\label{fig:4 boundary points 3}
		\end{figure}
		\item If we have $(\mu,\lambda,\nu)\in L_1$, the open line segment in $\mathring{\Delta}^2$ such that $\mu < \lambda = \nu$ (see \Cref{fig:simplex}). Then we are in the intersection of situations (1) and (2), where the rectangular cell $C(R)$ becomes a line segment of length 
		$$a_{34} - a_{23} = 2\log\left(\frac{\nu}{\mu}\right) = 2\log\left(\frac{\lambda}{\mu}\right) = a_{24} - a_{23}.$$ 
		Here $C(R)=\overline{C(R')^*}$ where $R'$ is an antipodal relation with graph 
		\begin{figure}[H]
			\centering
			\begin{tikzpicture}[scale=0.5,line cap=round,line join=round,>=triangle 45,x=1cm,y=1cm]
				\clip(1.5,1.5) rectangle (4.5,4.5);
				\draw [line width=0.4pt] (2,4)-- (4,4);
				\draw [line width=0.4pt] (2,2)-- (4,2);
				\draw [line width=0.4pt] (2,4)-- (4,2);
				\draw [line width=0.4pt] (2,2)-- (4,4);			
				\begin{scriptsize}
					\draw [fill=black] (2,4) circle (2.5pt);
					\draw[color=black] (1.6,4.1) node {$1$};
					\draw [fill=black] (4,4) circle (2.5pt);
					\draw[color=black] (4.4,4.1) node {$2$};
					\draw [fill=black] (2,2) circle (2.5pt);
					\draw[color=black] (1.6,1.9) node {$4$};
					\draw [fill=black] (4,2) circle (2.5pt);
					\draw[color=black] (4.4,1.9) node {$3$};
				\end{scriptsize}
			\end{tikzpicture}
		\end{figure} \noindent union of the graphs from situations (1) and (2). In this case, $\MM(Z, \rho)$ is a tree (see \Cref{fig:trees}).\\
		
		\item Similarly, if $(\mu,\lambda,\nu)\in L_2\subseteq \mathring{\Delta}^2$, that is $\lambda < \mu = \nu$, then the rectangular cell again becomes a line segment of length 
		$$a_{34} - a_{24} = 2\log\left(\frac{\nu}{\lambda}\right) = 2\log\left(\frac{\mu}{\lambda}\right) = a_{23} - a_{24},$$ 
		and $\MM(Z, \rho)$ is a tree (see \Cref{fig:trees}).\\
		
		\item Also, if $(\mu,\lambda,\nu)\in L_3\subseteq \mathring{\Delta}^2$, that is  $\nu < \mu = \lambda$, then the line segment is of length 
		$$a_{23} - a_{34} = 2\log\left(\frac{\mu}{\nu}\right) = 2\log\left(\frac{\lambda}{\nu}\right) = a_{24} - a_{34},$$ 
		with $\MM(Z, \rho)$ having the same tree structure as in the above two cases (see \Cref{fig:trees}).
		\begin{figure}[H]
			\centering
			\includegraphics[width=1\linewidth]{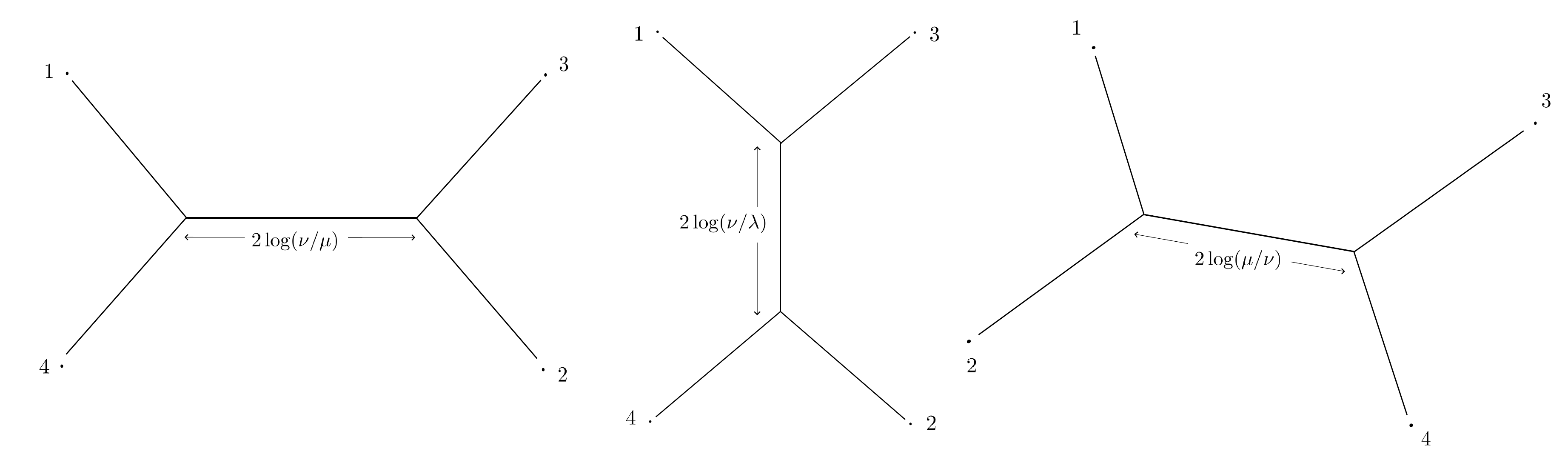}
			\caption{Structure of $\MM(Z,\rho)$ in last three cases respectively (from left to right)}
			\label{fig:trees}
		\end{figure}
		\item In the case $(\mu,\lambda,\nu)=O$ that is $\mu=\lambda=\nu$ the $C(R)$ collapses to a point. $\MM(Z,\rho)$ is star shaped tree.
		\begin{figure}[H]
			\centering
			\includegraphics[width=0.27\linewidth]{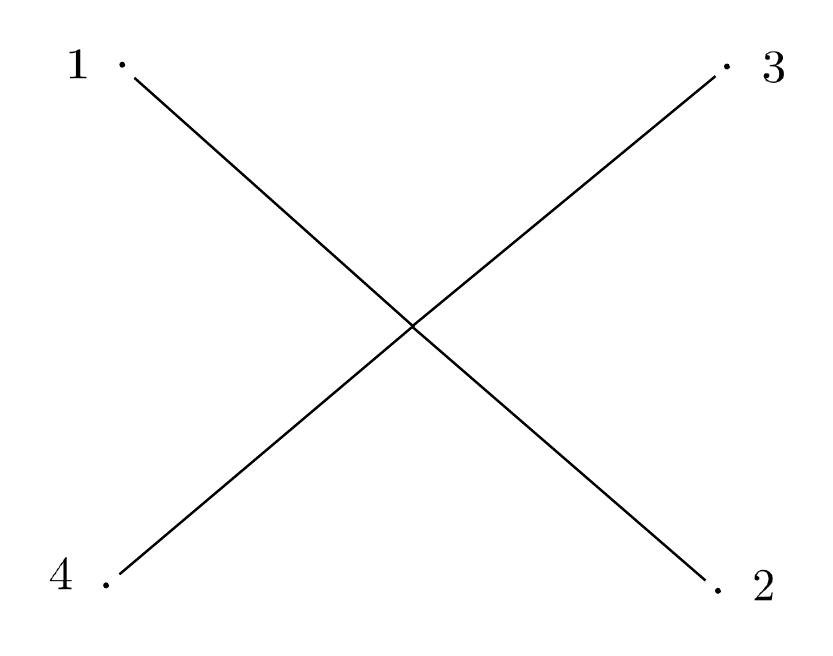}
			\caption{}
			\label{fig:trees star}
		\end{figure}
	\end{enumerate}
	Note, for any two points $(\mu, \lambda, \nu)$ and $(\mu', \lambda', \nu')$ in the union of the three open regions $B_1 \cup B_2 \cup B_3$, there exists a homeomorphism between the corresponding Moebius spaces $\MM(\rho)$ and $\MM(\rho')$. Thus, these spaces share the same Teichmuller spaces, up to a specific bijection, as seen in \eqref{identification of Teichmuller space}. The same holds for $(\mu, \lambda, \nu)$ and $(\mu', \lambda', \nu')$ in $L_1 \cup L_2 \cup L_3$.
	
	However, observe when the two points are from the same region, say $B_i$ (for some $i \in \{1, 2, 3\}$), they are homeomorphically equivalent, meaning there exists a homeomorphism between the two Moebius spaces where the boundary map is the identity. Furthermore, if $\MM(\rho)$ corresponds to a point in $B_i$ under the parametrization, then any $\MM(\rho') \in [\MM(\rho')]_H$ corresponds to a point in the same $B_i$. Hence, for any point $(\mu, \lambda, \nu) = \Phi(\MM(\rho)) \in B_i$, the homeomorphism equivalence class $[\MM(\rho)]_H$ is parametrized by $B_i$ via $\Phi$. The same is true for the open segments $L_i$: if $\Phi(\MM(\rho)) \in L_i$ for some $i \in \{1, 2, 3\}$, then $[\MM(\rho)]_H$ is parametrized by $L_i$. If $\Phi(\MM(\rho))=O$ then the homeomorphism class $[\MM(\rho)]_H$, is singleton.
	
	\begin{rmk}\label{Teichmuller space in the case of 4 points}  
		The big-Teichmuller space $\TT(Z)$ is the disjoint union of homeomorphism equivalence classes. For $Z$ containing four elements, the $2$-simplex $\Delta^2$ (which is the closure of the open simplex $\mathring{\Delta}^2$, the parameter space for $\TT(Z)$), admits a polyhedral complex structure, induced by the Homeomorphism equivalence classes. The relative interior of each polyhedron corresponds to the image of a Homeomorphism equivalence class under the parametrization $\Phi$. Consequently, the disjoint union of these relative interiors of polyhedra gives the open simplex $\mathring{\Delta}^2$ (see \Cref{fig:simplex}). Therefore, the Homeomorphism classes in $(\TT(Z),d_{\hbox{M\"ob}})$ are homeomorphic to open polyhedrons in Euclidean space of appropriate dimensions. As any Teichmuller space $\TT(\MM(\rho))$ can be parametrized by corresponding Homeomorphism equivalence class $[\MM(\rho)]_H$. In this case, there are three distinct Teichmuller spaces $\TT(\MM(\rho))$, up to identification \eqref{identification of Teichmuller space}, of dimensions $0,1$ and $2$ respectively.
	\end{rmk}   
	
	\bigskip
	
	\section{Open problems}\label{problems}
	
	We conclude with some questions which arise naturally in this context:
	
	\medskip
	
	\noindent (1) Characterization of the $l_{\infty}$-polyhedral complexes arising as maximal Gromov hyperbolic spaces: While \Cref{1st Thm} states that any maximal Gromov hyperbolic space with finite boundary can be realized as an $l_{\infty}$-polyhedral complex, it is natural to try and characterize exactly which polyhedral complexes occur as maximal Gromov hyperbolic spaces. One has here an obvious necessary condition, that the geometry of the polyhedral complex at infinity should be trivial, i.e. the complex should be given by attaching $n$ half-lines to a compact polyhedral complex. What additional conditions would be necessary and sufficient for the polyhedral complex to be a maximal Gromov hyperbolic space?
	
	\medskip
	
	\noindent (2) Possible polyhedral structures for maximal Gromov hyperbolic spaces with countable boundary: It would be interesting to see whether one has a polyhedral structure for maximal Gromov hyperbolic spaces $X$ with boundary $\partial X$ a countable compact space. At least for the case where $\partial X$ has finitely many non-isolated points one may hope to realize $X$ as a polyhedral complex (in $C(\partial X)$) with countably many cells.
	
	\medskip
	
	\noindent (3) Topology of the Teichmuller spaces: The central problem here is to understand, for a finite set $Z$ of cardinality $n$, the topology of the Teichmuller spaces ${\mathcal{T}}({\mathcal{M}}(\rho_0))$ as $\rho_0$ ranges over all antipodal functions on $Z$. For the case of $n =4$ discussed in the previous section, all the Teichmuller spaces occurring in this example are topologically cells of various dimensions, and moreover under the identification of the big-Teichmuller space with an open simplex, the various Teichmuller spaces are identified with polyhedra in the simplex, giving a polyhedral structure on the simplex. It would be interesting to see if the same holds true for all $n \geq 4$.
	
	\section*{Acknowledgements}
	\noindent The second author is supported by research fellowship from Indian Statistical Institute.
	\bigskip
	
	\bibliography{References}
	\bibliographystyle{alpha}
	
\end{document}